\documentclass{amsart}[14pt,a4paper]
\usepackage{geometry}
\usepackage{amsfonts}
\usepackage{hyperref}
\usepackage{amsmath}
\usepackage{galois}
\DeclareMathSizes{10}{10}{7}{5}
\usepackage{mathtools}
\numberwithin{equation}{section}

\usepackage{enumitem}
\usepackage{amssymb}

\numberwithin{equation}{section}

\usepackage{enumitem}
\usepackage{amssymb}
\usepackage{mathrsfs}

\newtheorem{theorem}{Theorem}

\newtheorem{remark}{Remark}
\newtheorem{lemma}{Lemma}

\begin{document}
	\title{On postcritical sets of quadratic polynomials with a neutral fixed point}
	\author{Hongyu Qu}
	\address[Hongyu Qu]{School of Science, Beijing University of Posts and Telecommunications, Beijing
		100786, P. R. China.}
	\address[Hongyu Qu]{Key Laboratory of Mathematics and Information Networks (Beijing University of Posts and Telecommunications), Ministry of Education, China.}
	%\renewcommand{\thefootnote}{\fnsymbol{footnote}}
	%\footnotetext[1]{Corresponding author, Email: hongyuqu2022@126.com}
	%\footnotetext[2]{The research work was supported by
		%	the National Natural Science Foundation of China (12301102 and 12171264) and the Fundamental Research Funds for the Central Universities. The research work was also supported by Key Laboratory of Mathematics and Information Networks (Beijing University of Posts and Telecommunications), Ministry of Education, China.}
	
	\maketitle
	\begin{abstract}
		The control of postcritical sets of quadratic polynomials with a neutral fixed point is a main ingredient in the remarkable work of Buff and Ch\'eritat
		to construct quadratic Julia sets with positive area. Based on the Inou-Shishikura theory, they obtained the control for the case of rotation numbers of bounded high type. Later, Cheraghi developed several elaborate analytic techniques based on Inou-Shishikura’s results and obtained the control for the case of rotation numbers of high type.
		In this paper, based on the pseudo-Siegel disk theory of Dudko and Lyubich, we obtained the control for the general case.
	\end{abstract}
	
	\section{introduction}
	Let the quadratic polynomial
	$$P_{\alpha}(z)=e^{2\pi i\alpha}z+z^2,$$
	where $0<\alpha<1$ is an irrational number with continued fraction expansion
	$$\alpha=[a_1,a_2,\cdots]=\frac{1}{a_1+\frac{1}{a_2+\ddots}}.$$
	According to \cite{Br} and \cite{Yoc95},
	$P_{\alpha}$ has a Siegel disk centering at $0$
	if and only if $\alpha$ is a Brjuno number.
	We denote by $\Delta_{\alpha}$ the Siegel disk of $P_{\alpha}$, in particular,
	if $\alpha$ is not a Brjuno number, $\Delta_{\alpha}=\{0\}$, and
	denote by $\mathcal{O}_{P_{\alpha}}$ the postcritical set
	of $P_{\alpha}$, that is
	$$\mathcal{O}_{P_{\alpha}}=\overline{\cup_{n=1}^{+\infty}\{P_{\alpha}^{\comp n}(c_0)\}}$$
	with $c_0:=-\frac{e^{2\pi i\alpha}}{2}$.
	
	The main purpose of this paper is to explore how the postcritical set $\mathcal{O}_{P_{\alpha}}$ depend on the irrational number $\alpha$. It is well known that
	the topological and geometric structure of $\mathcal{O}_{P_{\alpha}}$ itself depend on the arithmetic nature of $\alpha$ in a delicate fashion.
	For example,
	%In the past decades, the topology and geometry structure of $\mathcal{O}_{P_{\alpha}}$ had been studied sufficiently for some special types $\alpha$.
	due to the Douady-Ghys surgery, when $\alpha$ is of bounded type, $\mathcal{O}_{P_{\alpha}}$ is a quasicircle;
	due to the trans-quasiconformal surgery by Peterson and Zakeri, $\mathcal{O}_{P_{\alpha}}$ is a David circle and has zero area for almost everywhere $\alpha$ (see \cite{PZ});
	due to the theory of Inou-Shishikura, when $\alpha$ is of high type, (that is, all $a_j$ is greater than or equal to some sufficiently large positive integer ${\bf N}$), it was proved that (see \cite{CD},\cite{CD13},\cite{CD22},\cite{SY})
	\begin{itemize}
		\item if $\alpha$ is a Herman number, then $\mathcal{O}_{P_{\alpha}}$ is a Jordan curve and has zero area,
		\item if $\alpha$ is a Brjuno but not a Herman number, then $\mathcal{O}_{P_{\alpha}}$ is a one-side hairy Jordan curve and has zero area,
		\item if $\alpha$ is not a Brjuno number, then $\mathcal{O}_{P_{\alpha}}$ is a Cantor bouquet and has zero area.
	\end{itemize}
	Our main result is to prove that $\alpha\mapsto\mathcal{O}_{P_{\alpha}}$ has the following upper semi-continuous property:
	\begin{theorem}
		\label{T1}Given an irrational number $0<\alpha<1$, we let $\{\alpha_n\}_{n\geq1}$ be a sequence of irrational numbers between $0$ and $1$ with $\lim_{n\to+\infty}\alpha_n=\alpha$. Then for any $\epsilon>0$ and large enough $n$,
		the postcritical set $\mathcal{O}_{P_{\alpha_n}}$ is contained in the $\epsilon$-neighborhood of
		$\mathcal{O}_{P_{\alpha}}\cup\Delta_\alpha$, and hence the union $\mathcal{O}_{P_{\alpha_n}}\cup\Delta_{\alpha_n}$
		is also contained in the $\epsilon$-neighborhood of $\mathcal{O}_{P_{\alpha}}\cup\Delta_\alpha$.
	\end{theorem}
	Partial results of Theorem \ref{T1} were obtained in papers \cite{BC} and \cite{CD} based on the Inou-Shishikura theory\cite{IS}, where
	\cite{BC} for $\alpha, \alpha_n$ of bounded high type and \cite{CD} for $\alpha, \alpha_n$ of high type. In \cite{BC}, the result for the bounded type by Buff and Ch\'eritat is one of key points to construct quadratic Julia sets with positive area. Then by Buff and Ch\'eritat's method, Theorem \ref{T1} can help us construct more different types of quadratic Julia sets with positive area. The proof of Theorem \ref{T1} in this paper is based on the pseudo-Siegel disk theory of Dudko and Lyubich in \cite{DL}.
	
	\vspace{0.4cm}
	$\quad${\bf Acknowledgements.}\quad The author would like to thank Professor Gaofei Zhang for recommending and discussing the important paper \cite{DL} by Dudko and Lyubich. The research work was supported by
	the National Natural Science Foundation of China (12301102 and 12171264) and the Fundamental Research Funds for the Central Universities.
	
	\section{The conformal geometry and pseudo-Siegel disks}
	This section is mainly devoted to proving a key lemma (Lemma \ref{L1}) for the proof of Theorem \ref{T1} based on the theory of Dudko and Lyubich on pseudo-Siegel disks.
	%Many notations of this section are still as Dudko and Lyubich's convinence.
	\subsection{The extremal width of a family of curves.}
	Given a family of curves $\mathcal{F}$ in $\hat{\mathbb{C}}$, we denote by $\mathcal{W}(\mathcal{F})$ the extremal width of $\mathcal{F}$, that is the reciprocal of the extremal length of $\mathcal{F}$. A family of curves $\mathcal{G}$ is said to overflow $\mathcal{F}$ if every curve in $\mathcal{G}$ contains a subcurve which belongs to $\mathcal{F}$.
	If $\mathcal{G}$ overflows $\mathcal{F}$, then the extremal widths have the following relation:
	$$\mathcal{W}(\mathcal{G})\leq\mathcal{W}(\mathcal{F}).$$
	%By a lamination $\mathcal{L}$ we mean a family of pairwise disjoint simple rectifiable arcs such that supp $\mathcal{L}$ is measurable.
	By a (topological) rectangle in $\hat{\mathbb{C}}$ we mean a closed Jordan disk $\mathcal{R}$ together with a conformal map $h:\mathcal{R}\to E_x$ into the standard rectangle $E_x$, where
	$E_x:=[0,x]\times[0,1]\subset\mathbb{C}$. In this case, we define
	\begin{itemize}
		\item the base $\partial^{h,0}\mathcal{R}$ is $h^{-1}([0,x]\times\{0\})$;
		\item the roof $\partial^{h,1}\mathcal{R}$ is $h^{-1}([0,x]\times\{1\})$; \item every $h^{-1}(\{t\}\times[0,1])$, $0\leq t\leq x$ is a vertical curve of $\mathcal{R}$, in particular, $h^{-1}(\{\frac{x}{2}\}\times[0,1])$ is called the center arc of $\mathcal{R}$;
		\item the full family $\mathcal{F}^{full}(\mathcal{R})$ is the family of curves connecting the base and the roof in $\mathcal{R}$;
		\item the curve $h^{-1}(\{\frac{x}{2}\}\times[0,1])$ is also called the center arc of $\mathcal{F}^{full}(\mathcal{R})$;
		\item the extremal width of $\mathcal{R}$ is $\mathcal{W}(\mathcal{R}):=x$.
		It is well known that $$\mathcal{W}(\mathcal{R})=\mathcal{W}(\mathcal{F}^{full}(\mathcal{R})).$$
	\end{itemize}
	
	\subsection{The geometry of nests of tilings.} Consider a closed quasidisk $D\subset\mathbb{C}$.
	Fix a point $o$ in ${\rm int}(D)$ (the interior of $D$).
	Let $\phi$ be a conformal map from ${\rm int}(D)$ to $\mathbb{D}$ such that
	$\phi(o)=0$ and $\tilde{\phi}$ be the homeomorphism extension of $\phi$ to $D$. Then for any interval $I$ on $\partial D$, we define the combinatorial length $|I|_{\partial D}$ by the Euclidean length of $\tilde{\phi}(I)$. It is easy to check that the definition is independent of the choice of $\phi$. For any $b_1, b_2\in\partial D$, the combinatorial distance ${\rm dist}_{\partial D}(b_1,b_2)$ between $b_1$ and $b_2$ is defined by the minimum of combinatorial lengths of two intervals connecting $b_1$ and $b_2$ on $\partial D$. For any two nonempty subsets $A, B$ of $\partial D$, the combinatorial distance ${\rm dist}_{\partial D}(A,B)$ between $A$ and $B$ is defined by
	$$\inf_{a\in A, b\in B}{\rm dist}_{\partial D}(a,b).$$

	Let $\mathcal{T}=\{\mathcal{T}_n\}_{n\geq -1}$ be a system of finite partitions of $\partial D$ into finitely many closed intervals such that $\mathcal{T}_{n+1}$ is a refinement of $\mathcal{T}_n$. $\mathcal{T}$ is called a nest of tilings if
	\begin{itemize}
		\item the maximal diameter of intervals in $\mathcal{T}_n$ tends to $0$ as $n\to\infty$, and
		\item every interval in $\mathcal{T}_n$ for $n\geq m$ decomposes into at least two intervals of $\mathcal{T}_{n+2}$.
	\end{itemize}
	Similarly, a nest of tilings is defined for a closed quasiconformal arc.
	Intervals in $\mathcal{T}_n$ are also called intervals of level $n$ of $\mathcal{T}$. We denote by ${\rm EP}(\mathcal{T}_n)$ the set of endpoints of all intervals in $\mathcal{T}_n$.
	
	For any two intervals $I$ and $J$ of $\partial D$ with disjoint interiors and $I\cup J=\partial D$, if $\mathcal{T}=\{\mathcal{T}_n\}_{n\geq -1}$ is a nest of tilings on $I$ and $\mathcal{T}'=\{\mathcal{T}_n'\}_{n\geq -1}$ is a nest of tilings on $J$, then we define the union $\mathcal{T}\cup\mathcal{T}'$ as a nest of tilings on $\partial D$:
	$$\mathcal{T}\cup\mathcal{T}'=\{\mathcal{T}_n\cup\mathcal{T}_n'\}_{n\geq -1}.$$
	A nest of tilings $\mathcal{T}$ is said to have $M$-bounded combinatorics if $\mathcal{T}$ satisfies
	\begin{itemize}
		\item[(a)] $\mathcal{T}_{-1}$ has at most $M>1$ intervals;
		\item[(b)] for all $n\geq-1$, every interval in $\mathcal{T}_n$ consists of at most $M$ intervals in $\mathcal{T}_{n+1}$.
	\end{itemize}
	If $\mathcal{T}$ is only required to satisfy (b), then we call that $\mathcal{T}$ has post-$M$-bounded combinatorics.
	
	For any two disjoint intervals $I,J\subset\partial D$, we denote by
	\begin{itemize}
		\item $\mathcal{F}(I,J)$ the family of curves in $\hat{\mathbb{C}}$ connecting $I$ and $J$:
		$$\mathcal{F}(I,J):=\{\gamma:[0,1]\to\hat{\mathbb{C}}:\gamma(0)\in I,\ \gamma(1)\in J\};$$
		\item $\mathcal{F}^-(I,J)$ the family of curves in $D$ connecting $I$ and $J$:
		$$\mathcal{F}^-(I,J):=\{\gamma:[0,1]\to D:\gamma(0)\in I,\ \gamma(1)\in J\};$$
		\item $\mathcal{F}^+(I,J)$ the family of curves in $\hat{\mathbb{C}}\setminus{\rm int}(D)$ connecting $I$ and $J$:
		$$\mathcal{F}^+(I,J):=\{\gamma:[0,1]\to\hat{\mathbb{C}}\setminus{\rm int}(D):\gamma(0)\in I,\ \gamma(1)\in J\}.$$
		%\item $\mathcal{W}(\mathcal{F}(I,J))$ the extremal width of $\mathcal{F}(I,J)$, that is the reciprocal of the extremal length of $\mathcal{F}(I,J)$;
		%\item $\mathcal{W}(\mathcal{F}^-(I,J))$ the extremal width of $\mathcal{F}^-(I,J)$;
		%\item $\mathcal{W}(\mathcal{F}^+(I,J))$ the extremal width of $\mathcal{F}^+(I,J)$.
	\end{itemize}
	
	For any interval $I\in\mathcal{T}_n$, let $I_l, I_r\in\mathcal{T}_n$ be two its neighboring intervals. We denote by $[3I]^c$ the closure of $\partial D\setminus(I_l\cup I\cup I_r)$, that is
	$$[3I]^c=\overline{\partial D\setminus(I_l\cup I\cup I_r)},$$
	and define
	\begin{itemize}
		\item $\mathcal{F}^-_{3,\mathcal{T}}(I)$ to be the family of curves in $D$ connecting $I$ and $[3I]^c$;
		\item $\mathcal{F}^+_{3,\mathcal{T}}(I)$ to be the family of curves in $\hat{\mathbb{C}}\setminus D$ connecting $I$ and $[3I]^c$;
		\item $\mathcal{F}_{3,\mathcal{T}}(I)$ to be the family of curves in $\hat{\mathbb{C}}$ connecting $I$ and $[3I]^c$;
		\item $\mathcal{W}^{\pm}_{3,\mathcal{T}}(I)=\mathcal{W}(\mathcal{F}^{\pm}_{3,\mathcal{T}}(I))$; $\mathcal{W}_{3,\mathcal{T}}(I)=\mathcal{W}(\mathcal{F}_{3,\mathcal{T}}(I))$.
	\end{itemize}	
	
	A nest of tilings $\mathcal{T}$ is said to have essentially bounded outer geometry if there exists a constant $C$ such that for every interval $I$ of $\mathcal{T}$ we have $\mathcal{W}^+_{3,\mathcal{T}}(I)\leq C$. In this case, one also says that $\mathcal{T}$ has essentially $C$-bounded outer geometry. If moreover, $\mathcal{T}$ has (post-)$M$-bounded combinatorics, then one says that $\mathcal{T}$ has (post-) bounded outer geometry or (post-) $(C,M)$-bounded outer geometry. Similarly, essentially bounded inner geometry, (post-) bounded inner geometry, essentially bounded geometry, (post-) bounded geometry are defined.
	
	\begin{lemma}[\cite{DL}]
		\label{l11}For every pair $C, M$, there is a $K_{C,M}>1$ such that the following holds. Let $D$ be a closed quasidisk and $\mathcal{T}$ be a nest of tilings of $\partial D$. If $\mathcal{T}$ has $(C,M)$-bounded inner and outer geometries, then $D$ is a $K_{C,M}$ quasidisk.
	\end{lemma}
	
	%For a closed Jordan disk $D\subset\mathbb{C}$, consider two disjoint intervals $I,J\subset\partial D$. Let us view $\hat{\mathbb{C}}\setminus(I\cup J)$ as a Riemann surface; with respect to this Riemann surface both $I,J$ have two sides: the outer sides $I^+, J^+$ and the inner sides $I^-, J^-$. Ignoring the endpoints of $I,J$, a curve $\gamma:(0,1)\to\hat{\mathbb{C}}\setminus(I\cup J)$ lands at $\gamma(1)\in I^+$ if
	%$$\gamma[1-\epsilon,1)\subset\hat{\mathbb{C}}\setminus Z\ \forall\epsilon>0\ {\rm and}\ \lim_{\tau\to1}\gamma(\tau)=\gamma(1).$$
	%If for every $I\in\mathcal{T}$ we have $\mathcal{W}_3(I)\leq C$, then we say that the nest of tilings $\mathcal{T}$ has essentially $C$-bounded geometry.
	Consider two disjoint intervals $I,J\subset\partial D$.
	%Let us view $\hat{\mathbb{C}}\setminus(I\cup J)$ as a Riemann surface; with respect to this Riemann surface both $I,J$ have two sides: the outer sides $I^+, J^+$ and the inner sides $I^-, J^-$. Ignoring the endpoints of $I,J$, a curve $\gamma:(0,1)\to\hat{\mathbb{C}}\setminus(I\cup J)$ lands at $\gamma(1)\in I^+$ if
	%$$\gamma[1-\epsilon,1)\subset\hat{\mathbb{C}}\setminus Z\quad \forall\epsilon>0\ {\rm and}\ \lim_{\tau\to1}\gamma(\tau)=\gamma(1).$$
	%Let
	%\begin{itemize}
	%	\item $\mathcal{F}(I^+,J^+)$ be the family of curves in $\hat{\mathbb{C}}\setminus(I\cup J)$ connecting $I^+$ and $J^+$;
	%	\item $\mathcal{W}(I^+,J^+)=\mathcal{W}(\mathcal{F}(I^+,J^+))$ be the extremal width of $\mathcal{F}(I^+,J^+)$.
	%\end{itemize}
	Let $\mathcal{F}(I^+,J^+)$ be the family of curves in $\hat{\mathbb{C}}\setminus(I\cup J)$ such that for any curve $\gamma:[0,1]\to\hat{\mathbb{C}}$ in $\mathcal{F}(I^+,J^+)$, we have that
	\begin{itemize}
		\item $\gamma(0)\in I$, $\gamma(1)\in J$ and $\gamma(0,1)\subset\hat{\mathbb{C}}\setminus(I\cup J)$;
		\item
		$\gamma[1-\epsilon,1)\subset\hat{\mathbb{C}}\setminus D$ and $\gamma(0,\epsilon]\subset\hat{\mathbb{C}}\setminus D$ for some $0<\epsilon<1$.
	\end{itemize}
	Let $\mathcal{W}(I^+,J^+)=\mathcal{W}(\mathcal{F}(I^+,J^+))$ be the extremal width of $\mathcal{F}(I^+,J^+)$.
	
	For any disjoint intervals $I_1$ and $I_2$ on $\partial D$, if $L_1$ and $L_2$ are the closures of complement intervals between $I_1$ and $I_2$ respectively such that $I_1<L_1<I_2<L_2$ (``$<$'' means the clockwise order on $\partial D$ for intervals with disjoint interiors), then we denote by $\mathcal{F}_r(I_1^+,I_2^+)$ the set of curves $\gamma\in\mathcal{F}(I_1^+,I_2^+)$ such that $\gamma$ is disjoint from $L_2$ and by $\mathcal{F}_l(I_1^+,I_2^+)$ the set of curves $\gamma\in\mathcal{F}(I_1^+,I_2^+)$ such that $\gamma$ is disjoint from $L_1$;
	we set $\mathcal{W}_r(I_1^+,I_2^+):=\mathcal{W}(\mathcal{F}_r(I_1^+,I_2^+))$ and $\mathcal{W}_l(I_1^+,I_2^+):=\mathcal{W}(\mathcal{F}_l(I_1^+,I_2^+))$.
	
	\begin{lemma}[\cite{DL}]
		\label{l20}Let $I,J\subset\partial D$ be a pair of disjoint intervals with the complement interval $L$ such that $I<L<J$ and $|L|_{\partial D}<\frac{\pi}{2}$. If
		$K:=\mathcal{F}_r(I^+,J^+)-\mathcal{W}^+(I,J)\gg\log\lambda$ with $\lambda>2$, then there are intervals $J_1, I_1\subset L$ with
		$$|J_1|_{\partial D}<\frac{{\rm dist}_{\partial D}(I,J_1)}{\lambda},\ |I_1|_{\partial D}<\frac{{\rm dist}_{\partial D}(I_1,J)}{\lambda},\ I<J_1<I_1<J$$
		such that
		either $\mathcal{W}_r(I,J_1)$ or $\mathcal{W}_r(I_1,J)$ has width $\geq2K-O(\log\lambda)$.
	\end{lemma}
	The following lemma gives a relation between extremal width and combinatorial length (see [Lemma 2.5, \cite{DL}]).
	\begin{lemma}
		\label{l7131}If intervals $I_1,I_2,I_3,I_4\subseteq\partial D$ with disjoint interiors satisfy $I_1\cup I_2\cup I_3\cup I_4=\partial D$ and $I_1<I_2<I_3<I_4$, then
		\begin{itemize}
			\item $\mathcal{W}^-(I_1,I_3)\leq c$ {implies} $\min\{|I_1|_{\partial D},|I_3|_{\partial D}\}\preceq_c\min\{|I_2|_{\partial D},|I_4|_{\partial D}\}$;
			\item $\min\{|I_1|_{\partial D},|I_3|_{\partial D}\}\leq c\cdot\min\{|I_2|_{\partial D},|I_4|_{\partial D}\}$ {implies}  $\mathcal{W}^-(I_1,I_3)\preceq_c1$.
		\end{itemize}
		Here $x\preceq_cy$ means that $\frac{x}{y}$ has a universal bound depending on the parameter $c$.
	\end{lemma}
	
	%\begin{lemma}
	%Let $I,J\subset\partial D$ be a pair of intervals with $\lfloor I,J\rfloor<1/2$, and let be the complementary interval between $I,J$ with $I<L<J$. Set
	%$$K:=\mathcal{W}_L^o(I,J)-\mathcal{W}^+(I,J).$$
	%If $K\gg \log\lambda$ with $\lambda>2$, then there are intervals
	%$$J_1, I_1\subset L,\ |j_1|<\frac{{\rm dist}(I,J_1)}{\lambda},\ |I_1|<\frac{{\rm dist}(I_1,J)}{\lambda},\ I<J_1<I_1<J$$
	%such that
	%$$\mathcal{W}_{L_a}^o(I,J_1)\oplus\mathcal{W}_{L_b}^o(I_1,J)\geq K-O(\log\lambda),$$
	%where $L_a, L_b\subset L$ are the complementary intervals between $I, J_1$ and $I_1, J$ respectively.
	%In particular, either $\mathcal{W}_{L_a}^o(I,J_1)$ or $\mathcal{W}_{L_b}^o(I_1,J)$ has width $\geq2K-O(\log\lambda)$.
	%\end{lemma}
	The following two lemmas are stated for a nest of tilings. One can see corresponding results for pseudo-Siegel disks in \cite{DL}, see [Lemma\ 11.4\ (II),\ \cite{DL}] and [Lemma\ 11.6,\ \cite{DL}] for details.
	\begin{lemma}
		\label{l12}Assume that a nest of tilings $\mathcal{T}$ of the closed quasidisk $D$ has post-bounded inner and outer geometries. Then $\mathcal{T}$ has essentially bounded geometry.
	\end{lemma}
	%\begin{remark}
	%	{\rm In \cite{DL}, [Lemma\ 11.4\ (II),\ \cite{DL}] is stated for pseudo-Siegel disks, while here Lemma \ref{l12} is stated for a general frame. They are essentially the same, and here we also prove Lemme \ref{l12} by the same way as that of [Lemma\ 11.4\ (II),\ \cite{DL}].}
	%\end{remark}
	\begin{proof}
		We assume that $\mathcal{T}$ has post-($C,M$)-bounded inner and outer geometries. Next, we prove the lemma by contradiction. We suppose that
		for some $I\in\mathcal{T}_n$, $$\mathcal{W}_{3,\mathcal{T}}(I)=K\gg_{C,M}1.$$
		Evidently, $\mathcal{T}_n$ has at least $4$ intervals. We let $I_l$ and $I_r$ be two intervals of level $n$ adjacent to $I$ such that $I_r<I<I_l$.
		We take the given point $o$ in $D$ appropriately such that $|I|_{\partial D}>\frac{\pi}{2}$ and $|I|_{\partial D}\asymp|I_r|_{\partial D}\asymp|I_l\cup[3I]^c|_{\partial D}$. Then
		\begin{equation}
			\label{e2.12.1}\left|\partial D\setminus I\right|_{\partial D}<\frac{\pi}{2}
		\end{equation}
		and by Lemma \ref{l7131}, we have that there exists a constant $l_C>1$ depending on $C$ such that for any two adjacent intervals $J_1, J_2\subseteq I_r\cup I_l$ of the same level ($\geq n$), we have
		$$|J_1|_{\partial D}\asymp_{l_C}|J_2|_{\partial D},\ {\rm that\ is}\
		\frac{1}{l_C}<\frac{|J_1|_{\partial D}}{|J_2|_{\partial D}}<l_C.$$

		Let $I_{l,l}^{n+4}\subset I_l$ be the interval of level $n+4$ adjacent to $[3I]^c$; let $I_{l,r}^{n+4}\subset I_l$ be the interval of level $n+4$ adjacent to $I$; let $I_{r,l}^{n+4}\subset I_r$ be the interval of level $n+4$ adjacent to $I$; let $I_{r,r}^{n+4}\subset I_r$ be the interval of level $n+4$ adjacent to $[3I]^c$.
		Set
		$$\hat{I}:=I\cup I_{l,r}^{n+4}\cup I_{r,l}^{n+4}\
		{\rm and}\
		\widehat{[3I]^c}:=[3I]^c\cup I_{l,l}^{n+4}\cup I_{r,r}^{n+4}.$$
		It follows from [Lemma 2.8, \cite{DL}] that there exist two intervals $\tilde{I},\ \widetilde{[3I]^c}$ with
		$$I\subseteq\tilde{I}\subseteq\hat{I},\ [3I]^c\subseteq\widetilde{[3I]^c}\subseteq\widehat{[3I]^c}$$
		and
		a subfamily
		$\mathcal{R}$ of $\mathcal{F}(\widetilde{[3I]^c}^+,\tilde{I}^+)$ disjoint from the center arc $\gamma_{1/2}$ of $\mathcal{F}_{3,\mathcal{T}}^-(I)$ such that $$\mathcal{W}(\mathcal{R})\geq\mathcal{W}_{3,\mathcal{T}}(I)+O(C^-),$$
		where $C^-:=\mathcal{W}^-(I_l,I_r)+\mathcal{W}^-(I,\partial D\setminus\hat{I})+\mathcal{W}^-([3I]^c,\partial D\setminus\widehat{[3I]^c})$.
		%and $\gamma_{1/2}$ is the center arc of $\mathcal{F}_{3,\mathcal{T}}^-(I)$, that is the vertical curve in $\mathcal{F}_{3,\mathcal{T}}^-(I)$ splitting $\mathcal{F}_{3,\mathcal{T}}^-(I)$ into two rectangles with the half width $\frac{\mathcal{W}_{3,\mathcal{T}}^-(I)}{2}$.
		
		We divide $I$, $\overline{I_l\setminus I_{l,l}^{n+4}}$ and $\overline{I_r\setminus I_{r,r}^{n+4}}$ into the union of intervals of level $n+4$ as follows:
		$$I=\cup_{j=1}^kI_j^{n+4},\ \overline{I_l\setminus I_{l,l}^{n+4}}=\cup_{j=1}^{k_l}I_{l,j}^{n+4}\ {\rm and}\ \overline{I_r\setminus I_{r,r}^{n+4}}=\cup_{j=1}^{k_r}I_{r,j}^{n+4},$$
		where $3\leq k,k_l,k_r\leq M^4$ are positive integers, and $I_j^{n+4}$ ($1\leq j\leq k$), $I_{l,j}^{n+4}$ ($1\leq j\leq k_l$), $I_{r,j}^{n+4}$ ($1\leq j\leq k_r$) are intervals of level $n+4$.
		Then
		$\mathcal{F}^-(I_l,I_r)$ overflows $\mathcal{F}^-_{3,\mathcal{T}}(I_l)$; $\mathcal{F}^-(I,\partial D\setminus\hat{I})$ overflows $\mathcal{F}^-_{3,\mathcal{T}}(I^{n+4}_1)\cup\mathcal{F}^-_{3,\mathcal{T}}(I^{n+4}_2)\cup\cdots\cup\mathcal{F}^-_{3,\mathcal{T}}(I^{n+4}_k)$; $\mathcal{F}^-(\partial D\setminus\widehat{[3I]^c},[3I]^c)$ overflows $\left(\cup_{j=1}^{k}\mathcal{F}^-_{3,\mathcal{T}}(I^{n+4}_j)\right)\cup\left(\cup_{j=1}^{k_l}\mathcal{F}^-_{3,\mathcal{T}}(I^{n+4}_{l,j})\right)\cup\left(\cup_{j=1}^{k_r}\mathcal{F}^-_{3,\mathcal{T}}(I^{n+4}_{r,j})\right)$.
		It follows that
		$$\mathcal{W}^-(I_l,I_r)\leq\mathcal{W}^-_{3,\mathcal{T}}(I_l)\preceq_{C}1,$$
		$$\mathcal{W}^-(I,\partial D\setminus\hat{I})\leq\mathcal{W}^-_{3,\mathcal{T}}(I^{n+4}_1)+\mathcal{W}^-_{3,\mathcal{T}}(I^{n+4}_2)+\cdots+\mathcal{W}^-_{3,\mathcal{T}}(I^{n+4}_k)\preceq_{C,M}1$$
		and
		$$\mathcal{W}^-([3I]^c,\partial D\setminus\widehat{[3I]^c})\leq\sum_{j=1}^{k}\mathcal{W}^-_{3,\mathcal{T}}(I^{n+4}_j)+\sum_{j=1}^{k_l}\mathcal{W}^-_{3,\mathcal{T}}(I^{n+4}_{l,j})+\sum_{j=1}^{k_r}\mathcal{W}^-_{3,\mathcal{T}}(I^{n+4}_{r,j})\preceq_{C,M}1.$$
		Thus $C^-\preceq_{C,M}1$ and $$\mathcal{W}(\mathcal{R})\geq\mathcal{W}_{3,\mathcal{T}}(I)+O_{C,M}(1).$$
		Since $\mathcal{F}^+(\tilde{I},\widetilde{[3I]^c})$ overflows
		$\left(\cup_{j=1}^{k}\mathcal{F}^+_{3,\mathcal{T}}(I^{n+4}_j)\right)\cup\mathcal{F}^+_{3,\mathcal{T}}(I_{l,r}^{n+4})\cup\mathcal{F}^+_{3,\mathcal{T}}(I_{r,l}^{n+4})$, we have
		\begin{equation}
			\label{e2.12}\mathcal{W}(\mathcal{F}^+(\tilde{I},\widetilde{[3I]^c}))\leq
			\sum_{j=1}^{k}\mathcal{W}^+_{3,\mathcal{T}}(I^{n+4}_j)+\mathcal{W}^+_{3,\mathcal{T}}(I_{l,r}^{n+4})+\mathcal{W}^+_{3,\mathcal{T}}(I_{r,l}^{n+4})\preceq_{C,M}1.
		\end{equation}
		
		%For any disjoint intervals $I_1$ and $I_2$ on $\partial D$, if $L_1$ and $L_2$ are the closures of complement intervals between $L_1$ and $L_2$ respectively such that $I_1<L_1< I_2<L_2$, then
		%we denote by $\mathcal{F}_r(I_1^+,I_2^+)$ the set of curves $\gamma\in\mathcal{F}(I_1^+,I_2^+)$ such that $\gamma$ is disjoint from $L_2$ and by $\mathcal{F}_l(I_1^+,I_2^+)$ the set of curves $\gamma\in\mathcal{F}(I_1^+,I_2^+)$ such that $\gamma$ is disjoint from $L_1$;
		%we set $\mathcal{W}_r(I_1^+,I_2^+):=\mathcal{W}(\mathcal{F}_r(I_1^+,I_2^+))$ and $\mathcal{W}_l(I_1^+,I_2^+):=\mathcal{W}(\mathcal{F}_l(I_1^+,I_2^+))$.
		Recall that $\mathcal{F}_r(\widetilde{[3I]^c}^+,\tilde{I}^+)$ is the set of curves $\gamma\in\mathcal{F}(\widetilde{[3I]^c}^+,\tilde{I}^+)$ such that $\gamma$ is disjoint from $I_l\setminus(\widetilde{[3I]^c}\cup \tilde{I})$ and $\mathcal{F}_l(\widetilde{[3I]^c}^+,\tilde{I}^+)$ is the set of curves $\gamma\in\mathcal{F}(\widetilde{[3I]^c}^+,\tilde{I}^+)$ such that $\gamma$ is disjoint from $I_r\setminus(\widetilde{[3I]^c}\cup \tilde{I})$; $\mathcal{W}_r(\widetilde{[3I]^c}^+,\tilde{I}^+)=\mathcal{W}(\mathcal{F}_r(\widetilde{[3I]^c}^+,\tilde{I}^+))$
		and
		$\mathcal{W}_l(\widetilde{[3I]^c}^+,\tilde{I}^+)=\mathcal{W}(\mathcal{F}_l(\widetilde{[3I]^c}^+,\tilde{I}^+))$.
		
		%Let $\mathcal{R}_1$ be a sublamination of $\mathcal{R}\setminus\mathcal{F}^+(\hat{I},\widehat{[3I]^c})$ not intersecting $I_r\setminus(I_{r,l}^{n+4}\cup I_{r,r}^{n+4})$; let $\mathcal{R}_2$ be a sublamination of $\mathcal{R}\setminus\mathcal{F}^+(\hat{I},\widehat{[3I]^c})$ not intersecting $I_l\setminus(I_{l,l}^{n+4}\cup I_{l,r}^{n+4})$; let $\mathcal{R}_3$ be a sublamination of $\mathcal{R}\setminus\mathcal{F}^+(\hat{I},\widehat{[3I]^c})$ overflowing $\mathcal{F}^+(I_l,I_r)$.
		
		%$\mathcal{R}_1$ overflows $\mathcal{F}_r(\hat{I}^+,\widehat{[3I]^c}^+)$;
		%$\mathcal{R}_2$ overflows $\mathcal{F}_l(\hat{I}^+,\widehat{[3I]^c}^+)$;
		%$\mathcal{R}_3$ overflows $\mathcal{F}^+(I_l,I_r)$.
		
		%$\mathcal{R}\setminus\mathcal{F}^+(\hat{I},\widehat{[3I]^c})$ overflows $\mathcal{R}_1\cup\mathcal{R}_2\cup\mathcal{R}_3$.
		%By the definition of $\mathcal{R}$
		Since $\mathcal{R}\subset\mathcal{F}(\widetilde{[3I]^c}^+,\tilde{I}^+)$ is disjoint from the center arc $\gamma_{1/2}$, we have that
		$\mathcal{R}$
		%$\setminus\mathcal{F}^+(\hat{I},\widehat{[3I]^c})$
		overflows $\mathcal{F}_r(\widetilde{[3I]^c}^+,\tilde{I}^+)\cup\mathcal{F}_l(\widetilde{[3I]^c}^+,\tilde{I}^+)\cup\mathcal{F}^+(I_l,I_r)\cup\mathcal{F}^+(I_r,I_l)$.
		Since $\mathcal{F}^+(I_l,I_r)$ overflows $\mathcal{F}^+_{3,\mathcal{T}}(I_l)$, we have $$\mathcal{W}(\mathcal{F}^+(I_l,I_r))\leq\mathcal{W}^+_{3,\mathcal{T}}(I_l)\preceq_{C}1.$$
		Then
		\begin{align*}
			\mathcal{W}_r(\widetilde{[3I]^c}^+,\tilde{I}^+)+\mathcal{W}_l(\widetilde{[3I]^c}^+,\tilde{I}^+)&\geq\mathcal{W}(\mathcal{R})-2\mathcal{W}(\mathcal{F}^+(I_l,I_r))\\
			%&\geq\mathcal{W}(\mathcal{R})-\mathcal{W}(\mathcal{F}^+(I_l,I_r))\\
			&\geq\mathcal{W}_{3,\mathcal{T}}(I)+O_{C,M}(1).
		\end{align*}
		Thus
		\begin{equation*}
			\mathcal{W}_r(\widetilde{[3I]^c}^+,\tilde{I}^+)\geq0.45K\ {\rm or}\
			\mathcal{W}_l(\widetilde{[3I]^c}^+,\tilde{I}^+)\geq0.45K.
		\end{equation*}
		Without loss of generality, we assume
		\begin{equation}
			\label{e2.20}\mathcal{W}_r(\widetilde{[3I]^c}^+,\tilde{I}^+)\succeq K.
		\end{equation}
		Set $R_{C,M}:=2Ml_C^{2M+1}$.
		Then it follows from (\ref{e2.12.1}), (\ref{e2.12}), (\ref{e2.20}) and Lemma \ref{l20} that there exist two intervals $I_1',I_1''\subseteq I_r\setminus(\widetilde{[3I]^c}\cup \tilde{I})$ with $\widetilde{[3I]^c}<I_1'<I_1''<\tilde{I}$, $|I_1'|_{\partial D}<\frac{1}{R_{C,M}}{\rm dist}_{\partial D}(I_1',\widetilde{[3I]^c})$, $|I_1''|_{\partial D}<\frac{1}{R_{C,M}}{\rm dist}_{\partial D}(I_1'',\tilde{I})$ such that
		\begin{equation}
			\label{e2.21}\mathcal{W}_r(\widetilde{[3I]^c}^+,I_1'^+)\geq2\mathcal{W}_r(\widetilde{[3I]^c}^+,\tilde{I}^+)+O_{C,M}(1)\ {\rm or}\
			\mathcal{W}_r(I_1''^+,\tilde{I}^+)\geq2\mathcal{W}_r(\widetilde{[3I]^c}^+,\tilde{I}^+)+O_{C,M}(1).
		\end{equation}
		If $\mathcal{W}_r(\widetilde{[3I]^c}^+,I_1'^+)\geq2\mathcal{W}_r(\widetilde{[3I]^c}^+,\tilde{I}^+)+O_{C,M}(1)$, then we set $I_0:=\widetilde{[3I]^c}$ and $I_1:=I_1'$; if $\mathcal{W}_r(I_1''^+,\tilde{I}^+)\geq2\mathcal{W}_r(\widetilde{[3I]^c}^+,\tilde{I}^+)+O_{C,M}(1)$, then we set $I_0:=I_1''$ and $I_1:=\tilde{I}$.
		Thus by (\ref{e2.20}) and (\ref{e2.21}), we have
		\begin{equation}
			\label{e2.10}\mathcal{W}_r(I_0^+,I_1^+)\succeq\frac{3}{2}K.
		\end{equation}
		
		Without loss of generality, we assume that $I_0=I_1''$ and $I_1=\tilde{I}$.
		Then $[3I]^c<I_0<I_1$.
		We let $t$ (necessary $>n$) be the smallest positive integer such that there is an interval $J'$ of level $t$ between $I_0$ and $I_1$ with $I_0<J'<I_1$. Then $[3I]^c<I_0<J'<I_1$.
		We decompose $I_r$ into intervals of level $t$, written as $I_r=\cup_{j=1}^{k_t}I_{r,j}^t$, with
		$I_{r,1}^t<I_{r,2}^t<\cdots<I_{r,k_t}^t$.
		%Let $I_{r,0}^t\subseteq[3I]^c$ be the interval of level $t$ adjacent to $I_{r,1}^t$. Then $I_{r,0}^t<I_{r,1}^t<\cdots<I_{r,k_t}^t$.
		Then there exist positive integers $j_1,j_0,j_2$ with
		$1\leq j_1\leq j_0\leq j_2\leq k_t$ such that $I_{r,j_0}^t=J'$, $I_{r,j_1}^t\cap I_0\not=\emptyset$, $I_{r,j_2}^t\cap I_1\not=\emptyset$ and $I_{r,j}^t\cap (I_1\cup I_0)=\emptyset$ for all $j_1<j<j_2$.
		The smallest property of $t$ gives $j_2-j_1+1\leq 2M$.
		Observe that $\sum_{j=j_1}^{j_2}|I_{r,j}^t|_{\partial D}\geq{\rm dist}_{\partial D}(I_0,I_1)>R_{C,M}|I_0|_{\partial D}$.
		Then
		$\max_{j_1\leq j\leq j_2}\left\{|I_{r,j}^t|_{\partial D}\right\}>\frac{R_{C,M}|I_0|_{\partial D}}{2M}$ and hence
		$|I_{r,j_1}^t|_{\partial D}>\frac{R_{C,M}|I_0|_{\partial D}}{2Ml_C^{2M}}=l_C|I_0|_{\partial D}$.
		If $j_1=1$, then $I_0\subseteq I_{r,j_1}^t$; if $j_1>1$, then
		$|I_{r,j_1-1}^t|_{\partial D}>|I_0|_{\partial D}$ and hence
		$I_0\subseteq I_{r,j_1-1}^t\cup I_{r,j_1}^t$.
		It follows that $\mathcal{F}^+(I_0,I_1)$ overflows $\mathcal{F}^+_{3,\mathcal{T}}(I_{r,j_1}^t)$ or $\mathcal{F}^+_{3,\mathcal{T}}(I_{r,j_1-1}^t)\cup\mathcal{F}^+_{3,\mathcal{T}}(I_{r,j_1}^t)$:
		if $j_1<j_0$, then $\mathcal{F}^+(I_0,I_1)$ overflows $\mathcal{F}^+_{3,\mathcal{T}}(I_{r,j_1}^t)$ or $\mathcal{F}^+_{3,\mathcal{T}}(I_{r,j_1-1}^t)\cup\mathcal{F}^+_{3,\mathcal{T}}(I_{r,j_1}^t)$;
		if $j_1=j_0$, then $I_0\subseteq I_{r,j_1-1}^t$ and hence $\mathcal{F}^+(I_0,I_1)$ overflows $\mathcal{F}^+_{3,\mathcal{T}}(I_{r,j_1}^t)$ or $\mathcal{F}^+_{3,\mathcal{T}}(I_{r,j_1-1}^t)\cup\mathcal{F}^+_{3,\mathcal{T}}(I_{r,j_1}^t)$.
		Thus
		\begin{equation}
			\label{e7121}\mathcal{W}^+(I_0,I_1)\leq\left\{
			\begin{matrix}
				\mathcal{W}^+_{3,\mathcal{T}}(I_{r,j_1}^t),&j_1=1\\
				\mathcal{W}^+_{3,\mathcal{T}}(I_{r,j_1-1}^t)+\mathcal{W}^+_{3,\mathcal{T}}(I_{r,j_1}^t),&j_1>1
			\end{matrix}\right.=O_{C}(1).
		\end{equation}
		
		Let $L$ be the complement interval between $I_0$ and $I_1$ such that
		$I_0<L<I_1$. By (\ref{e2.10}) and (\ref{e7121}),
		applying Lemma \ref{l20} to $I_0$, $L$ and $I_1$,
		we have that there exist two intervals $I_2',I_2''\subseteq L$ with $I_0<I_2'<I_2''<I_1$, $|I_2'|_{\partial D}<\frac{1}{R_{C,M}}{\rm dist}_{\partial D}(I_2',I_0)$, $|I_2''|_{\partial D}<\frac{1}{R_{C,M}}{\rm dist}_{\partial D}(I_2'',I_1)$ such that
		$$\mathcal{W}_r(I_0^+,I_2'^+)\geq2\mathcal{W}_r(I_0^+,I_1^+)+O_{C,M}(1)\ {\rm or}\
		\mathcal{W}_r(I_2''^+,I_1^+)\geq2\mathcal{W}_r(I_0^+,I_1^+)+O_{C,M}(1).$$
		If $\mathcal{W}_r(I_0^+,I_2'^+)\geq2\mathcal{W}_r(I_0^+,I_1^+)+O_{C,M}(1)$, then we set $I_2:=I_2'$; if $\mathcal{W}_r(I_2''^+,I_1^+)\geq2\mathcal{W}_r(I_0^+,I_1^+)+O_{C,M}(1)$, then we set $I_2:=I_2''$.
		Thus we have
		$$\mathcal{W}_r(I_0^+,I_2^+)\succeq(\frac{3}{2})^2K\ {\rm and}\
		|I_2|_{\partial D}<\frac{1}{R_{C,M}}{\rm dist}_{\partial D}(I_0,I_2)$$
		or
		$$\mathcal{W}_r(I_2^+,I_1^+)\succeq(\frac{3}{2})^2K\ {\rm and}\
		|I_2|_{\partial D}<\frac{1}{R_{C,M}}{\rm dist}_{\partial D}(I_1,I_2).$$
		By the second inequalities of these two formulas respectively, similar to (\ref{e7121}), we have
		$$\mathcal{W}^+(I_0,I_2)\preceq_{C,M}1$$
		or
		$$\mathcal{W}^+(I_2,I_1)\preceq_{C,M}1.$$
		
		Similarly, we can obtain that for all $n\geq1$,
		\begin{align}
			\label{e2.2}\mathcal{W}_r(I_{j_{n,1}}^+,I_n^+)\succeq(\frac{3}{2})^nK\ {\rm and}\
			|I_n|_{\partial D}<\frac{1}{R_{C,M}}{\rm dist}_{\partial D}(I_{j_{n,1}},I_n)
		\end{align}
		or
		\begin{align}
			\label{e2.211}\mathcal{W}_r(I_n^+,I_{j_{n,2}}^+)\succeq(\frac{3}{2})^nK\ {\rm and}\
			|I_n|_{\partial D}<\frac{1}{R_{C,M}}{\rm dist}_{\partial D}(I_n,I_{j_{n,2}})
		\end{align}
		for some two nonnegative integers $j_{n,1}(<n)$ and $j_{n,2}(<n)$.
		Moreover, we have that for all $n\geq1$,
		$$\mathcal{W}^+(I_{j_{n,1}},I_n)\preceq_{C,M}1\ {\rm(Thanks\ to\ (\ref{e2.2}))}$$
		or
		$$\mathcal{W}^+(I_n,I_{j_{n,2}})\preceq_{C,M}1\ {\rm(Thanks\ to\ (\ref{e2.211}))}.$$
		
		Since $D$ is a quasidisk, there exists a quasiconformal map $\phi$ from $\hat{\mathbb{C}}$ to itself mapping $D$ to $\hat{\mathbb{C}}\setminus\overline{\mathbb{D}}$ and $\partial D$ to the unit circle $\mathbb{S}^1$. Then for all $j\geq0$,
		$\phi(I_j)$ are intervals in $\mathbb{S}^1$, say, arcs in $\mathbb{S}^1$.
		For any intervals $I,J$ in $\mathbb{S}^1$, we denote by $\mathcal{F}_{\mathbb{D}}(I,J)$ the family of all curves in $\hat{\mathbb{C}}$ connecting $I$ and $J$; we denote by $\mathcal{F}^-_{\mathbb{D}}(I,J)$ the family of all curves in $\mathbb{D}$ connecting $I$ and $J$. Extremal widths of $\mathcal{F}_{\mathbb{D}}(I,J)$ and $\mathcal{F}^-_{\mathbb{D}}(I,J)$ are written as $\mathcal{W}_{\mathbb{D}}(I,J)$ and $\mathcal{W}^-_{\mathbb{D}}(I,J)$ respectively.
		The symmetry principle gives
		\begin{equation}
			\label{e2.1}\mathcal{W}_{\mathbb{D}}(I,J)=2\mathcal{W}^-_{\mathbb{D}}(I,J).
		\end{equation} 
		It is easy to see that
		$$\phi(\mathcal{F}^+(I_{j_{n,1}},I_n))=\mathcal{F}^-_{\mathbb{D}}(\phi(I_{j_{n,1}}),\phi(I_n)),\ \phi(\mathcal{F}(I_{j_{n,1}},I_n))=\mathcal{F}_{\mathbb{D}}(\phi(I_{j_{n,1}}),\phi(I_n)),$$
		$$\phi(\mathcal{F}^+(I_n,I_{j_{n,2}}))=\mathcal{F}^-_{\mathbb{D}}(\phi(I_n),\phi(I_{j_{n,2}}))\ {\rm and}\ \phi(\mathcal{F}(I_n,I_{j_{n,2}}))=\mathcal{F}_{\mathbb{D}}(\phi(I_n),\phi(I_{j_{n,2}})).$$
		Thus
		$$\mathcal{W}(\phi(\mathcal{F}^+(I_{j_{n,1}},I_n)))=\mathcal{W}^-_{\mathbb{D}}(\phi(I_{j_{n,1}}),\phi(I_n)),\ \mathcal{W}(\phi(\mathcal{F}(I_{j_{n,1}},I_n)))=\mathcal{W}_{\mathbb{D}}(\phi(I_{j_{n,1}}),\phi(I_n)),$$
		$$\mathcal{W}(\phi(\mathcal{F}^+(I_n,I_{j_{n,2}})))=\mathcal{W}^-_{\mathbb{D}}(\phi(I_n),\phi(I_{j_{n,2}}))\ {\rm and}\ \mathcal{W}(\phi(\mathcal{F}(I_n,I_{j_{n,2}})))=\mathcal{W}_{\mathbb{D}}(\phi(I_n),\phi(I_{j_{n,2}})).$$
		Since
		for all $n\geq1$,
		$$\mathcal{W}^+(I_{j_{n,1}},I_n)\preceq_{C,M}1\ {\rm(Thanks\ to\ (\ref{e2.2}))}$$
		or
		$$\mathcal{W}^+(I_n,I_{j_{n,2}})\preceq_{C,M}1\ {\rm(Thanks\ to\ (\ref{e2.211}))},$$
		we have that
		$$\{\mathcal{W}^-_{\mathbb{D}}(\phi(I_{j_{n,1}}),\phi(I_n))\}_{n\geq1}\ {\rm(Thanks\ to\ (\ref{e2.2}))}$$
		or
		$$\{\mathcal{W}^-_{\mathbb{D}}(\phi(I_n),\phi(I_{j_{n,2}}))\}_{n\geq1}\ {\rm(Thanks\ to\ (\ref{e2.211}))}$$
		are bounded. By (\ref{e2.1}), we have that
		$$\{\mathcal{W}_{\mathbb{D}}(\phi(I_{j_{n,1}}),\phi(I_n))\}_{n\geq1}\ {\rm(Thanks\ to\ (\ref{e2.2}))}$$
		or
		$$\{\mathcal{W}_{\mathbb{D}}(\phi(I_n),\phi(I_{j_{n,2}}))\}_{n\geq1}\ {\rm(Thanks\ to\ (\ref{e2.211}))}$$
		are bounded. It follows that
		$$\{\mathcal{W}(\phi(\mathcal{F}(I_{j_{n,1}},I_n)))\}_{n\geq1}\ {\rm(Thanks\ to\ (\ref{e2.2}))}$$
		or
		$$\{\mathcal{W}(\phi(\mathcal{F}(I_n,I_{j_{n,2}})))\}_{n\geq1}\ {\rm(Thanks\ to\ (\ref{e2.211}))}$$
		are bounded. Thus
		$$\{\mathcal{W}(\mathcal{F}(I_{j_{n,1}},I_n))\}_{n\geq1}\ {\rm(Thanks\ to\ (\ref{e2.2}))}$$
		or
		$$\{\mathcal{W}(\mathcal{F}(I_n,I_{j_{n,2}}))\}_{n\geq1}\ {\rm(Thanks\ to\ (\ref{e2.211}))}$$
		are bounded. 
		This contradicts with the first inequality of (\ref{e2.2}) or (\ref{e2.211}).

	\end{proof}
	
	Let $D$ be a closed quasidisk and $\mathcal{T}$ be a nest of tilings of $\partial D$ with at least two intervals of level $-1$ having post-$(C,M)$-bounded inner and outer geometries. Let $\gamma$ be a geodesic with endpoints contained in ${\rm EP}(\mathcal{T}_{-1})$ with respect to the hyperbolic metric on ${\rm int}(D)$. The geodesic $\gamma$ decomposes $D$ into two closed quasidisks $D_1$ and $D_2$ such that $D=D_1\cup D_2$ and $D_1\cap D_2=\gamma$.
	Let $\mathcal{T}^1$ and $\mathcal{T}^2$ be the restrictions of $\mathcal{T}$ to $\overline{\partial D_1\setminus\gamma}$ and $\overline{\partial D_2\setminus\gamma}$ respectively.
	
	\begin{lemma}
		\label{l0}There is a nest of tilings $\mathcal{T}^3$ on $\gamma$ and a constant $\eta_{C,M}$ determined by $C, M$ such that both $\mathcal{T}^1\cup\mathcal{T}^3$ and $\mathcal{T}^2\cup\mathcal{T}^3$ have post-$(\eta_{C,M},M)$-bounded inner and outer geometries.
		Moreover, if $\mathcal{T}^1$ {\rm(resp. $\mathcal{T}^2$)} has at most $M$ intervals of level $-1$, then $\mathcal{T}^1\cup\mathcal{T}^3$ {\rm(resp. $\mathcal{T}^2\cup\mathcal{T}^3$)} has $(\eta_{C,M},2M)$-bounded inner and outer geometries.
	\end{lemma}
	\begin{proof}
		Without loss of generality, we assume that the number of intervals of level $-1$ of $\mathcal{T}^1$ is less than or equal to that of $\mathcal{T}^2$.
		Let $n_0$ be the lowest level of $\mathcal{T}^1$ having at least two intervals. We denote by $I^{n_0}$ and $J^{n_0}$ two intervals of level $n_0$ of $\mathcal{T}^1$ adjacent to $\gamma$.
		Let $\phi$ be a conformal map from ${\rm int}(D)$ to the unit disk $\mathbb{D}$ such that $\phi$ maps $\gamma$ to the real interval $(-1,1)$, maps $D_1$ to the upper half unit disk and maps $D_2$ to the lower half unit disk, and the Euclidean lengths satisfy 
		\begin{equation}
			\label{e7141}|\phi(I^{n_0})|_{\rm Euc}=|\phi(J^{n_0})|_{\rm Euc}.
		\end{equation}
		We denote also by $\phi$ the homeomorphism extension of $\phi$ to $D$.
		Then the image $\phi(\mathcal{T})$ has post-$(C,M)$-bounded inner geometry with respect to $\mathbb{D}$. By the definition of $\phi$, we have that $\phi(\mathcal{T}^1)$ is a nest of tilings of the upper half unit circle $S^+$ and $\phi(\mathcal{T}^2)$ is a nest of tilings of the lower half unit circle $S^-$.
		
		We define a homeomorphism
		$$h:[-1,1]\to S^+, x\mapsto e^{-i\pi(x-1)/2}.$$
		The inverse image $\mathcal{T}^{3,*}:=h^{-1}(\phi(\mathcal{T}^1))$ is a nest of tilings of $[-1,1]$. Set $\mathcal{T}^3:=\phi^{-1}(\mathcal{T}^{3,*})$, which is a nest of tilings of $\gamma$.
		It follows easily from the definition of $\mathcal{T}^3$ that both $\mathcal{T}^1\cup\mathcal{T}^3$ and $\mathcal{T}^2\cup\mathcal{T}^3$ have post-$M$-bounded combinatorics.
		
		Next we prove that both $\mathcal{T}^1\cup\mathcal{T}^3$ and $\mathcal{T}^2\cup\mathcal{T}^3$ have essentially bounded inner and outer geometries.
		We only give the proof for the case $\mathcal{T}^1\cup\mathcal{T}^3$, while the case $\mathcal{T}^2\cup\mathcal{T}^3$ is similar.
		To do it, we firstly prove that $\mathcal{T}^{3,*}\cup\phi(\mathcal{T}^1)$ (resp. $\mathcal{T}^{3,*}\cup\phi(\mathcal{T}^2)$) has essentially bounded inner geometry for the upper half unit disk (resp. the lower half unit disk).
		We denote by $\mathbb{D}^+$ and $\mathbb{D}^-$ the upper half unit disk and the lower half unit disk respectively.
		We set
		$$T(z):=\frac{(\frac{z+1}{1-z})^2-i}{(\frac{z+1}{1-z})^2+i}=\frac{(z+1)^2-i(1-z)^2}{(z+1)^2+i(1-z)^2}.$$
		The rational function $T$ conformally maps $\mathbb{D}^+$ to $\mathbb{D}$ and can extend to a homeomorphism from $\overline{\mathbb{D}^+}$ to $\overline{\mathbb{D}}$.
		Given any pairwise different complex numbers $a,b,c\in\overline{\mathbb{D}^+}$, a simple computation gives
		$$\left|\frac{T(c)-T(b)}{T(a)-T(b)}\right|=\left|\frac{c-b}{a-b}\right|\cdot\left|\frac{1-bc}{1-ba}\right|\cdot\left|\frac{(a+1)^2+i(a-1)^2}{(c+1)^2+i(c-1)^2}\right|.$$
		Observe that for any $a,c\in\overline{\mathbb{D}^+}$, we have
		\begin{align*}
			8&\geq|(a+1)^2+i(a-1)^2|\\
			&=|(1+i)(a-i+\sqrt{2}i)(a-i-\sqrt{2}i)|\\
			&\geq\sqrt{2}\cdot(\sqrt{2}-1)\cdot\sqrt{2}\\
			&=2(\sqrt{2}-1)
		\end{align*}
		and
		\begin{align*}
			8&\geq|(c+1)^2+i(c-1)^2|\\
			&=|(1+i)(c-i+\sqrt{2}i)(c-i-\sqrt{2}i)|\\
			&\geq\sqrt{2}\cdot(\sqrt{2}-1)\cdot\sqrt{2}\\
			&=2(\sqrt{2}-1).
		\end{align*}
		Then we have
		$$\left|\frac{(a+1)^2+i(a-1)^2}{(c+1)^2+i(c-1)^2}\right|\asymp1.$$
		Let
		$b=|b|e^{i\alpha_1}$ ($0\leq\alpha_1\leq\pi$) and
		$a=|a|e^{i\alpha_2}$ ($0\leq\alpha_2\leq\pi$). Then
		\begin{align*}
			|1-ba|&=|1-|a||b|e^{i(\alpha_1+\alpha_2)}|\\
			&=\left|1-|a||b|\cos(\alpha_1+\alpha_2)-i|a||b|\sin(\alpha_1+\alpha_2)\right|\\
			&=
			\sqrt{1+|a|^2|b|^2-2|a||b|\cos(\alpha_1+\alpha_2)}\\
			&\geq
			\sqrt{1+|a|^2|b|^2-2|a||b|\cdot|\cos\alpha_1|}\\
			&=\left\{\begin{matrix}
				\sqrt{|b|^2\left(|a|-\frac{|\cos\alpha_1|}{|b|}\right)^2+\sin^2\alpha_1},&|b|\not=0\\1,&|b|=0\end{matrix}\right.\\
			&\geq\left\{\begin{matrix}|\sin\alpha_1|,&|b|\geq|\cos\alpha_1|\\\min\{|1-b|,|1+b|\},&|b|<|\cos\alpha_1|\end{matrix}\right..
		\end{align*}
		Observe that if $|b|\geq|\cos\alpha_1|$, then we have
		\begin{align*}
			\min\{|1-b|,|1+b|\}&=\sqrt{1+|b|^2-2|b|\cdot|\cos\alpha_1|}\\
			&\leq\sqrt{2-2|\cos\alpha_1|}\\
			&=\left\{\begin{matrix}2\sin\frac{\alpha_1}{2},&0\leq\alpha_1\leq\frac{\pi}{2}\\2\cos\frac{\alpha_1}{2},&\frac{\pi}{2}<\alpha_1\leq\pi\end{matrix}\right.\\
			&\leq\sqrt{2}|\sin\alpha_1|.
		\end{align*}
		Thus
		$$|1-ba|\geq\min\left\{\frac{|1-b|}{\sqrt{2}},\frac{|1+b|}{\sqrt{2}}\right\}.$$
		We set $k=\frac{c-b}{a-b}$. Then we have
		$$\left|\frac{1-bc}{1-ba}\right|=\left|k+(1-k)\frac{1-b^2}{1-ba}\right|\leq|k|+|1-k|\cdot\sqrt{2}(1+|b|)\leq5(1+|k|).$$
		Similarly, we have
		$$\left|\frac{1-ba}{1-bc}\right|\leq5(1+\frac{1}{|k|}).$$
		Then
		$$\frac{1}{5(1+\frac{1}{|k|})}\leq\left|\frac{1-bc}{1-ba}\right|\leq5(1+|k|).$$
		Thus
		\begin{equation}
			\label{e10}\frac{|k|^2}{5(|k|+1)}\preceq\left|\frac{T(c)-T(b)}{T(a)-T(b)}\right|\preceq5|k|(1+|k|).
		\end{equation}
		
		For any four different intervals $I_1<I_2<I_3<I_4$ of level $n$ of	$\mathcal{T}^{3,*}\cup\phi(\mathcal{T}^1)$ such that $I_2$ is exactly the complement interval between $I_1$ and $I_3$, by Lemma \ref{l7131} and (\ref{e7141}),
		the Euclidean lengths satisfy $\min\{|I_2|_{\rm Euc},|[3I_2]^c|_{\rm Euc}\}\preceq_C\min\{|I_1|_{\rm Euc},|I_3|_{\rm Euc}\}$.
		Then (\ref{e10}) gives $\min\{|T(I_2)|_{\rm Euc},|T([3I_2]^c)|_{\rm Euc}\}\preceq_C\min\{|T(I_1)|_{\rm Euc},|T(I_3)|_{\rm Euc}\}$. Applying Lemma \ref{l7131} to $T(\mathcal{T}^{3,*}\cup\phi(\mathcal{T}^1))$, we have that the nest of tilings $T(\mathcal{T}^{3,*}\cup\phi(\mathcal{T}^1))$ has essentially bounded inner geometry with respect to $\mathbb{D}$, and hence $\mathcal{T}^{3,*}\cup\phi(\mathcal{T}^1)$ has essentially bounded inner geometry with respect to $\mathbb{D}^+$.
		Similarly, by Lemma \ref{l7131} and (\ref{e10}) $\mathcal{T}^{3,*}\cup\phi(\mathcal{T}^2)$ has essentially bounded inner geometry with respect to $\mathbb{D}^-$.
		Thus we have that
		\begin{itemize}
			\item[(a)] $\mathcal{T}^1\cup\mathcal{T}^3$ has essentially bounded inner geometry with respect to $D_1$;
			\item[(b)] $\mathcal{T}^2\cup\mathcal{T}^3$ has essentially bounded inner geometry with respect to $D_2$.
		\end{itemize}
		
		Next we prove that $\mathcal{T}^1\cup\mathcal{T}^3$ has essentially bounded outer geometry with respect to $D_1$. For any interval $I$ of level $n\geq-1$ of $\mathcal{T}^1\cup\mathcal{T}^3$, we consider 
		$\mathcal{W}^+_{3,\mathcal{T}^1\cup\mathcal{T}^3}(I)$.
		
		\noindent{\bf Case 1:} If $I$ belongs to $\mathcal{T}^1$ and has no common endpoint with $\gamma$, then $\mathcal{F}_{3,\mathcal{T}^1\cup\mathcal{T}^3}^+(I)$ overflows $\mathcal{F}^+_{3,\mathcal{T}}(I)$ and hence $\mathcal{W}^+_{3,\mathcal{T}^1\cup\mathcal{T}^3}(I)\leq\mathcal{W}^+_{3,\mathcal{T}}(I)\preceq_C1$.
		
		\noindent{\bf Case 2:} If $I$ belongs to $\mathcal{T}^3$ and has no common endpoint with $\gamma$, then $\mathcal{F}_{3,\mathcal{T}^1\cup\mathcal{T}^3}^+(I)$ overflows $\mathcal{F}^-_{3,\mathcal{T}^2\cup\mathcal{T}^3}(I)$ and hence by (b) we have $\mathcal{W}^+_{3,\mathcal{T}^1\cup\mathcal{T}^3}(I)\leq\mathcal{W}^-_{3,\mathcal{T}^2\cup\mathcal{T}^3}(I)\preceq_C1$.
		
		\noindent{\bf Case 3:} If $I$ belongs to $\mathcal{T}^1$ and has a common endpoint with $\gamma$, then $\mathcal{F}_{3,\mathcal{T}^1\cup\mathcal{T}^3}^+(I)$ overflows $\mathcal{F}_{3,\mathcal{T}}(I)\cup\mathcal{F}^-_{3,\mathcal{T}^2\cup\mathcal{T}^3}(I')$, where the interval $I'$ of $\mathcal{T}^2$, adjacent to $I$, has the same level as $I$. Thus by Lemma \ref{l12} and (b) we have $$\mathcal{W}^+_{3,\mathcal{T}^1\cup\mathcal{T}^3}(I)\leq\mathcal{W}_{3,\mathcal{T}}+\mathcal{W}^-_{3,\mathcal{T}^2\cup\mathcal{T}^3}(I')\preceq_{C,M}1.$$
		\noindent{\bf Case 4:} If $I$ belongs to $\mathcal{T}^3$ and has a common endpoint with $\gamma$, then $\mathcal{F}_{3,\mathcal{T}^1\cup\mathcal{T}^3}^+(I)$ overflows $\mathcal{F}^-_{3,\mathcal{T}^2\cup\mathcal{T}^3}(I)\cup\mathcal{F}^-_{3,\mathcal{T}^2\cup\mathcal{T}^3}(I'')\cup\mathcal{F}_{3,\mathcal{T}}(I'')\cup\mathcal{F}^-_{3,\mathcal{T}^2\cup\mathcal{T}^3}(I''')$, where
		the intervals $I'', I'''$ of $\mathcal{T}^2$ has the same level as $I$ and satisfy
		\begin{itemize}
			\item $I''$ is adjacent to $I$;
			\item $I'''$ is adjacent to $I''$.
		\end{itemize}
		Thus by Lemma \ref{l12} and (b) we have $$\mathcal{W}^+_{3,\mathcal{T}^1\cup\mathcal{T}^3}(I)\leq\mathcal{W}^-_{3,\mathcal{T}^2\cup\mathcal{T}^3}(I)+\mathcal{W}^-_{3,\mathcal{T}^2\cup\mathcal{T}^3}(I'')+\mathcal{W}_{3,\mathcal{T}}(I'')+\mathcal{W}^-_{3,\mathcal{T}^2\cup\mathcal{T}^3}(I''')\preceq_{C,M}1.$$
		Thus for any interval $I$ of $\mathcal{T}^1\cup\mathcal{T}^3$, Cases $1-4$ give $\mathcal{W}^+_{3,\mathcal{T}^1\cup\mathcal{T}^3}(I)\preceq_{C,M}1.$ This implies that there is a constant $\eta_{C,M}$ determined by $C, M$ such that $\mathcal{T}^1\cup\mathcal{T}^3$ has post-$(\eta_{C,M},M)$-bounded inner and outer geometries. %If $\mathcal{T}^1$ has at most $M$ intervals of level $-1$, then $\mathcal{T}^1\cup\mathcal{T}^3$ has $(\eta_{C,M},M)$-bounded inner and outer geometries.
		Moreover, if $\mathcal{T}^1$ {\rm(resp. $\mathcal{T}^2$)} has at most $M$ intervals of level $-1$, then $\mathcal{T}^1\cup\mathcal{T}^3$ {\rm(resp. $\mathcal{T}^2\cup\mathcal{T}^3$)} has $(\eta_{C,M},2M)$-bounded inner and outer geometries.
	\end{proof}
	
	\subsection{Pseudo-Siegel disks.\label{s2.3}}
	Assume that $\alpha=[a_1,a_2,\cdots]$ is an eventually golden-mean rotation number, that is, there exists a positive integer $j_0$ such that
	$a_j=1$ for all $j\geq j_0$. Let $\frac{p_n}{q_n}$ ($q_n>0, p_n$ are integers) be the $n-$th rational approximation of $\alpha$. Then
	$$\frac{p_n}{q_n}=[a_1,a_2,\cdots,a_n].$$
	Due to the Douady-Ghys surgery, the quadratic polynomial $P_{\alpha}(z)=e^{2\pi i\alpha}z+z^2$ has a Siegel disk $\Delta_{\alpha}$ centering at $0$ whose boundary is a quasicircle containing a critical point.
	
	For $n\geq-1$, we denote by ${\rm CP}_n$ the set of critical points of $f^{\comp q_{n+1}}$. The diffeo-tiling $\mathfrak{D}_n(\alpha)$ of level $n$ is the partition of $\partial\Delta_{\alpha}$ induced by ${\rm CP}_n$: every interval in $\mathfrak{D}_n(\alpha)$ is the closure of a component of $\partial\Delta_{\alpha}\setminus{\rm CP}_n$. For $n=-1$, the tilings consists of a single ``interval''.
	
	%Let $\mathfrak{D}_n(\alpha)$ be the diffeo-tiling of level $n$.
	We assign each $I_n\in\mathfrak{D}_n(\alpha)$ a geodesic $\tau_n(I_n)$ with respect to the hyperbolic metric on $\hat{\mathbb{C}}\setminus\overline{\Delta_{\alpha}}$ such that two endpoints of $\tau_n(I_n)$ are in $I_n\cap{\rm CP}_{n+1}\setminus{\rm CP}_n$. Let $F_n$ be the filling-in of $$\partial\Delta_{\alpha}\cup\left(\bigcup_{I_n\in\mathfrak{D}_n(\alpha)}\tau_n(I_n)\right),$$
	that is the union of $\partial\Delta_{\alpha}\cup\left(\bigcup_{I_n\in\mathfrak{D}_n(\alpha)}\tau_n(I_n)\right)$
	and all bounded components of $$\mathbb{C}\setminus\left(\partial\Delta_{\alpha}\cup\left(\bigcup_{I_n\in\mathfrak{D}_n(\alpha)}\tau_n(I_n)\right)\right).$$
	Such $F_n$ is called a geodesic pre-filling-in of $\Delta_{\alpha}$ of level $n$ associated with $\tau_n$.
	
	A nest $\{\hat{\Delta}_{\alpha}^n\}_{n\geq-1}$ is called a geodesic filling-in of $\Delta_{\alpha}$ if
	\begin{itemize}
		\item $\hat{\Delta}_{\alpha}^n=\overline{\Delta_{\alpha}}$ for $n\gg1$;
		\item for all $n\geq0$,
		$$\hat{\Delta}_{\alpha}^{n-1}=\hat{\Delta}_{\alpha}^n\ {\rm or}\ \hat{\Delta}_{\alpha}^{n-1}=\hat{\Delta}_{\alpha}^n\cup F_{n-1},$$
		where $F_{n-1}$ is a geodesic pre-filling-in of $\Delta_{\alpha}$ of level $n-1$.
	\end{itemize}
	Each $\hat{\Delta}_{\alpha}^n$ is called a geodesic filling-in of $\Delta_{\alpha}$ of level $n$.
	It is easy to see that $\hat{\Delta}_{\alpha}^n$ is a closed quasidisk.
	
	In \cite{DL}, Dudko and Lyubich proved that there exists a geodesic filling-in of $\Delta_{\alpha}$: $\{\hat{\Delta}_{\alpha}^n\}_{n\geq-1}$ such that the following properties hold:
	\begin{itemize}
		\item[(1)] {\bf Injectivity:} for all $n\geq-1$, $f^{\comp q_{n+1}}$ is univalent on ${\rm int}(\hat{\Delta}_{\alpha}^n)$; (refer to Assumption $6$ and Lemma 4.6 of \cite{DL} for details)
		\item[(2)] {\bf Uniform a priori bounds:} there exist universal constants $C, M$ (not depending on $\alpha$) such that for all $n\geq-1$, there exists a nest of tilings $\mathcal{T}=\{\mathcal{T}_k\}_{k\geq-1}$ of $\partial\hat{\Delta}_{\alpha}^n$ with post-$(C,M)$-bounded inner and outer geometries; moreover,
		there is an one-to-one correspondence between intervals of $\mathcal{T}_{-1}$ and intervals of  $\mathfrak{D}_n(\alpha)$ such that each interval in $\mathcal{T}_{-1}$ has the same endpoints as the corresponding interval in $\mathfrak{D}_n(\alpha)$ and they are homotopic in $\mathbb{C}^*$. (refer to Section $11.4$ of \cite{DL} for details)
	\end{itemize}
	%A pseudo-Siegel disk $\hat{\Delta}_{\alpha}^n$ of level $n$ is obtained from $\overline{\Delta_{\alpha}}$ by filling-in deep parts of parabolic fjords of levels $\geq n$.
	In this case, each such $\hat{\Delta}_{\alpha}^n$ is called a pseudo-Siegel disk of level $n$. For any $I_n(\alpha)\in\mathfrak{D}_n(\alpha)$, the corresponding interval in $\mathcal{T}_{-1}$ is called the projection of $I_n(\alpha)$ on $\partial\hat{\Delta}_{\alpha}^n$.
	
	Let $\hat{\Delta}_{\alpha}^n$ be a pseudo-Siegel disk of level $n\geq1$. For any $I_n(\alpha)\in\mathfrak{D}_n(\alpha)$,
	we let $\hat{I}_n(\alpha)$ be the projection of $I_n(\alpha)$ on $\partial\hat{\Delta}_{\alpha}^n$ and
	$\gamma_n(\alpha)$ be the geodesic connecting two endpoints of $\hat{I}_n(\alpha)$
	with respect to the hyperbolic metric on ${\rm int}(\hat{\Delta}_{\alpha}^n)$.
	Then we have the following lemma.
	\begin{lemma}
		\label{L1}There is an absolute constant $\bf{K}$ $>1$ such that
		for any $I_n(\alpha)\in\mathfrak{D}_n(\alpha)$, the closed curve $\hat{I}_n(\alpha)+\gamma_n(\alpha)$ is a $\bf{K}$-quasicircle.
	\end{lemma}
	\begin{proof}
		Let $\mathcal{T}$ be the nest of tilings of $\partial\hat{\Delta}_{\alpha}^n$ as described in (2).
		Let $\mathcal{T}^1$ be the restriction of $\mathcal{T}$ to $\hat{I}_n(\alpha)$. Then by Lemma \ref{l0} there exists a nest of tilings $\mathcal{T}^3$ on $\gamma_n(\alpha)$ and a constant $\eta_{C,M}$ determined by $C,M$ such that $\mathcal{T}^1\cup\mathcal{T}^3$ has post-($\eta_{C,M},M$)-bounded inner and outer geometries. Moreover,
		since $\mathcal{T}^1$ has only one interval of level $-1$, we have that
		$\mathcal{T}^1\cup\mathcal{T}^3$ has ($\eta_{C,M},2M$)-bounded inner and outer geometries. Then the lemma follows immediately from Lemma \ref{l11}.	
	\end{proof}

	\section{The proof of Theorem \ref{T1}}
	Let the continued fraction expansion of $\alpha_n$ be
	$$\alpha_n=[a_1^{(n)},a_2^{(n)},\cdots,a_j^{(n)},\cdots],$$
	where all $a_j^{(n)}$ ($j\geq1$) are positive integers.
	For any positive integer $N$,
	replacing all entries after the $N$-th term by $1$, we obtain
	$$\alpha_n^{(N)}:=[a_1^{(n)},a_2^{(n)},\cdots,a_N^{(n)},1,1,\cdots].$$
	For any positive integer $m$, we let $n(m)$ and $N(m)$ be two positive integers depending on $m$ such that the following properties hold:
	\begin{itemize}
		\item $N(m)\gg m$;
		\item $a_j^{(n(m))}=a_j$ for $1\leq j\leq m+1$;
		\item $\left|P_{\alpha_{n(m)}^{(N(m))}}^{\comp j}(-\frac{e^{2\pi i\alpha_{n(m)}^{(N(m))}}}{2})-P_{\alpha}^{\comp j}(-\frac{e^{2\pi i\alpha}}{2})\right|<\frac{1}{m}$ for all $j\in\{0,1,2,\cdots,q_{m+1}\}$,
		where $q_{m+1}$ is the denominator of the ${\rm(m+1)-th}$ rational approximation of $\alpha$.
	\end{itemize}
	Note that for any $m$, the above conditions will be satisfied as soon as $n(m)$ and $N(m)$ are large enough.
	
	For any nonempty subsets $A$ and $B$ of $\hat{\mathbb{C}}$, we denote by ${\rm dist}_{\rm H}(A,B)$ the Hausdorff semi-distance of $A$ and $B$ with respect to the spherical metric:
	$${\rm dist}_{\rm H}(A,B)=\sup_{x\in A}{\rm dist}_{\hat{\mathbb{C}}}(x,B).$$
	Then we have the following claim:
	
	\vspace{0.2cm}
	\noindent{\bf Claim:}
	$$\sup_{I_m(\alpha_{n(m)}^{(N(m))})\in\mathfrak{D}_m(\alpha_{n(m)}^{(N(m))})}{\rm dist_H}\left(\hat{I}_m(\alpha_{n(m)}^{(N(m))})+\gamma_m(\alpha_{n(m)}^{(N(m))}), \mathcal{O}_{P_{\alpha}}\cup\Delta_{\alpha}\right)\to0\ {\rm as}\ m\to+\infty,$$
	where $\gamma_m(\alpha_{n(m)}^{(N(m))})$ is the hyperbolic geodesic as the same endpoints as $\hat{I}_m(\alpha_{n(m)}^{(N(m))})$ on ${\rm int}(\hat{\Delta}_{\alpha_{n(m)}^{(N(m))}}^m)$ (here $\hat{\Delta}_{\alpha_{n(m)}^{(N(m))}}^m$ is a pseudo-Siegel disk of level $m$).
	%The diameter of $I_m(\alpha_{n(m)}^{(N(m))})+\gamma_m(\alpha_{n(m)}^{(N(m))})$ with respect to the spherical metric converges to $0$ as $m$ goes to $+\infty$.
	
	\vspace{0.4cm}
	\noindent In fact, if the claim does't hold, without loss of generality, we suppose that
	\begin{equation}
		\label{e7181}\lim_{m\to+\infty}{\rm dist_H}\left(\hat{I}_m(\alpha_{n(m)}^{(N(m))})+\gamma_m(\alpha_{n(m)}^{(N(m))}), \mathcal{O}_{P_{\alpha}}\cup\Delta_{\alpha}\right)=A>0.
	\end{equation}
	For large enough $m$, we let $a_m$ and $b_m$ be two endpoints of $I_m(\alpha_{n(m)}^{(N(m))})$. Then $P_{\alpha_{n(m)}^{(N(m))}}^{\comp q_{m+1}}(a_m)$ or $P_{\alpha_{n(m)}^{(N(m))}}^{\comp q_{m+1}}(b_m)$ belongs to $I_m(\alpha_{n(m)}^{(N(m))})$.
	%, where $q_{m+1}$ is the denominator of the ${\rm(m+1)-th}$ rational approximation of $\alpha$.
	Furthermore, we have that $P_{\alpha_{n(m)}^{(N(m))}}^{\comp q_{m+1}}(a_m)$ or $P_{\alpha_{n(m)}^{(N(m))}}^{\comp q_{m+1}}(b_m)$ belongs to $\hat{I}_m(\alpha_{n(m)}^{(N(m))})$. This follows from the following fact:

\vspace{0.2cm}
\noindent {\bf Fact:} \emph{For any interval $I\in\mathfrak{D}_t(\alpha_{n(m)}^{(N(m))})$ {\rm(}$t\geq m${\rm)},
	we decompose $I$ into some intervals in $\mathfrak{D}_{t+1}(\alpha_{n(m)}^{(N(m))})$ as follows: $I=\cup_{j=1}^{k_t}I_j^{t+1}$, where
	$I_j^{t+1}\in\mathfrak{D}_{t+1}(\alpha_{n(m)}^{(N(m))})$ {\rm(}$1\leq j\leq k_t${\rm)} and $I_1^{t+1}<I_2^{t+1}<\cdots<I_{k_t}^{t+1}$.
	If $P_{\alpha_{n(m)}^{(N(m))}}^{\comp j}\left(-\frac{e^{2\pi i\alpha_{n(m)}^{(N(m))}}}{2}\right)$, $1\leq j<q_{m+1}$ belongs to $I$, then $P_{\alpha_{n(m)}^{(N(m))}}^{\comp j}\left(-\frac{e^{2\pi i\alpha_{n(m)}^{(N(m))}}}{2}\right)$ must belong to $I_1^{t+1}\cup I_{k_t}^{t+1}$.}
	
	\vspace{0.2cm}
	\noindent Observe that as $m$ goes to $+\infty$, both $P_{\alpha_{n(m)}^{(N(m))}}^{\comp q_{m+1}}(a_m)$ and $P_{\alpha_{n(m)}^{(N(m))}}^{\comp q_{m+1}}(b_m)$ converge to $\mathcal{O}_{P_{\alpha}}$. Thus
	$$\lim_{m\to+\infty}{\rm dist}_{\hat{\mathbb{C}}}(\hat{I}_m(\alpha_{n(m)}^{(N(m))}),\mathcal{O}_{P_{\alpha}})=0.$$
	This implies that
	$$\liminf_{m\to+\infty}\ {\rm diam}_{\hat{\mathbb{C}}}\left(\hat{I}_m(\alpha_{n(m)}^{(N(m))})+\gamma_m(\alpha_{n(m)}^{(N(m))})\right)\geq A.$$
	It follows from Lemma \ref{L1} that $\hat{I}_m(\alpha_{n(m)}^{(N(m))})+\gamma_m(\alpha_{n(m)}^{(N(m))})$ is a $\bf{K}$-quasicircle.
	Thus without loss of generality, we can suppose that $\hat{I}_m(\alpha_{n(m)}^{(N(m))})+\gamma_m(\alpha_{n(m)}^{(N(m))})$
	converges uniformly to a $\bf{K}$-quasicircle $\Gamma$ with ${\rm diam_{\hat{\mathbb{C}}}}(\Gamma)\geq A$ with respect to the spherical metric as $m\to+\infty$.
	Then the bounded region bounded by $\hat{I}_m(\alpha_{n(m)}^{(N(m))})+\gamma_m(\alpha_{n(m)}^{(N(m))})$
	converges to a Jordan domain $\mathcal{D}_{\Gamma}$ not containing $\infty$ having the boundary $\Gamma$ in the sense of Carathéodory.
	By (1) of Subsection \ref{s2.3} and Hurwitz theorem, we have that for all $t$, the restriction of $P_{\alpha}^{\comp t}$ to $\mathcal{D}_{\Gamma}$ is univalent.
	
	\vspace{0.2cm}
	\noindent{\bf Case 1:} If $\mathcal{D}_{\Gamma}$ intersects the Julia set of $P_{\alpha}$, then
	for large enough $t$, the critical point $c_0\in P_{\alpha}^{\comp t}(\mathcal{D}_{\Gamma})$ and this contradicts the fact that
	$P_{\alpha}^{\comp(t+1)}$ is univalent on $\mathcal{D}_{\Gamma}$.
	
	\vspace{0.2cm}
	\noindent{\bf Case 2:} Assume $\mathcal{D}_{\Gamma}\subseteq P_{\alpha}^{\comp-k}(\Delta_{\alpha})\setminus\Delta_{\alpha}$.
	Then $P_{\alpha}^{\comp k}(\mathcal{D}_{\Gamma})\subseteq\Delta_{\alpha}$.
	We choose an Euclidean ball $\mathbb{B}$ with $0\not\in\overline{\mathbb{B}}\subseteq P_{\alpha}^{\comp k}(\mathcal{D}_{\Gamma})\subseteq\Delta_{\alpha}$.
	Let $\mathbb{B}_{-k}$ be the connected component of $P_{\alpha}^{-k}(\mathbb{B})$ which is contained in $\mathcal{D}_{\Gamma}$.
	For any $n\geq0$, we have $P_{\alpha}^{\comp n}(\mathbb{B})\subseteq\Delta_{\alpha}$ and hence $P_{\alpha}^{\comp n}(\mathbb{B})\cap\mathbb{B}_{-k}=\emptyset$.
	
	We fix a sufficiently large $n_0$ such that there exists an Euclidean ball $\mathbb{E}$ with $\overline{\mathbb{E}}\subseteq P_{\alpha}^{\comp n_0}(\mathbb{B})\cap\mathbb{B}$ satisfying the following conditions:
	\begin{itemize}
		\item $P_{\alpha}^{\comp-k}(P_{\alpha}^{\comp n_0}(\mathbb{B}))$ has a connected component which is compactly contained in $\mathcal{D}_{\Gamma}$;
		\item $P_{\alpha}^{\comp n_0}(\mathbb{E})\subseteq P_{\alpha}^{\comp n_0}(\mathbb{B})\cap\mathbb{B}$;
		\item $\cup_{j=0}^{j=n_0}P_{\alpha}^{\comp j}(\mathbb{E})$ separates $0$ and $\infty$.
	\end{itemize}
	We assume that $m$ is large enough. Then $P_{\alpha_{n(m)}^{(N(m))}}$ satisfies
	\begin{itemize}
		\item[(a)] $P_{\alpha_{n(m)}^{(N(m))}}^{\comp-k}(P_{\alpha_{n(m)}^{(N(m))}}^{\comp n_0}(\mathbb{B}))$ has a connected component which is contained in ${\rm int}(\hat{\Delta}_{\alpha_{n(m)}^{(N(m))}}^m)$;
		\item[(b)] $P_{\alpha_{n(m)}^{(N(m))}}^{\comp n_0}(\mathbb{E})\subseteq P_{\alpha_{n(m)}^{(N(m))}}^{\comp n_0}(\mathbb{B})\cap\mathbb{B}$;
		\item[(c)] $\cup_{j=0}^{j=n_0}P_{\alpha_{n(m)}^{(N(m))}}^{\comp j}(\mathbb{E})$ separates $0$ and $\infty$;
		\item[(d)] $\cup_{j=0}^{j=n_0}P_{\alpha_{n(m)}^{(N(m))}}^{\comp j}(\mathbb{E})\subseteq\Delta_{\alpha}$;
		\item[(e)] $P_{\alpha_{n(m)}^{(N(m))}}^{\comp-k}(\mathbb{B})$ has a connected component which is compactly contained in $\mathcal{D}_{\Gamma}$, written as $\tilde{\mathbb{B}}_{-k}$;
		\item[(f)] $\tilde{\mathbb{B}}_{-k}\subseteq{\rm int}(\hat{\Delta}_{\alpha_{n(m)}^{(N(m))}}^m)$;
		\item[(g)] for any $0\leq j\leq n_0$, $P_{\alpha_{n(m)}^{(N(m))}}^{\comp j}(\mathbb{B})\subseteq\Delta_{\alpha}$ and hence $P_{\alpha_{n(m)}^{(N(m))}}^{\comp j}(\mathbb{B})\cap\tilde{\mathbb{B}}_{-k}=\emptyset$.
	\end{itemize}
	By $(1)$ of Subsection \ref{s2.3}, we have that $P_{\alpha_{n(m)}^{(N(m))}}^{\comp k}$ is univalent on ${\rm int}(\hat{\Delta}_{\alpha_{n(m)}^{(N(m))}}^m)$.
	Then by (f) we have
	$$P_{\alpha_{n(m)}^{(N(m))}}^{\comp-k}(\mathbb{B})\cap{\rm int}(\hat{\Delta}_{\alpha_{n(m)}^{(N(m))}}^m)=\tilde{\mathbb{B}}_{-k}.$$
	We let $\tilde{\mathbb{B}}_{-k,t}$ ($1\leq t\leq2^k-1$) be the other $2^k-1$ connected components of $P_{\alpha_{n(m)}^{(N(m))}}^{-k}(\mathbb{B})$. Then for all $1\leq t\leq2^k-1$,
	$$\tilde{\mathbb{B}}_{-k,t}\cap{\rm int}(\hat{\Delta}_{\alpha_{n(m)}^{(N(m))}}^m)=\emptyset.$$
	Since $P_{\alpha_{n(m)}^{(N(m))}}^{\comp(n_0-k)}(\mathbb{B})\cap\tilde{\mathbb{B}}_{-k}=\emptyset$ (Thanks to (g)) and $P_{\alpha_{n(m)}^{(N(m))}}^{\comp n_0}(\mathbb{B})\cap\mathbb{B}\not=\emptyset$ (Thanks to (b)),
	we have that $P_{\alpha_{n(m)}^{(N(m))}}^{\comp(n_0-k)}(\mathbb{B})\cap\tilde{\mathbb{B}}_{-k,t_0}\not=\emptyset$ for some $1\leq t_0\leq2^k-1$ and hence $P_{\alpha_{n(m)}^{(N(m))}}^{\comp(n_0-k)}(\mathbb{B})\nsubseteq{\rm int}(\hat{\Delta}_{\alpha_{n(m)}^{(N(m))}}^m)$.
	Since $P_{\alpha_{n(m)}^{(N(m))}}^{\comp k}$ is univalent on ${\rm int}(\hat{\Delta}_{\alpha_{n(m)}^{(N(m))}}^m)$, by (a) we have $$P_{\alpha_{n(m)}^{(N(m))}}^{\comp(n_0-k)}(\mathbb{B})\cap{\rm int}(\hat{\Delta}_{\alpha_{n(m)}^{(N(m))}}^m)=\emptyset.$$
	Again, observe $\Delta_{\alpha_{n(m)}^{(N(m))}}\subseteq{\rm int}(\hat{\Delta}_{\alpha_{n(m)}^{(N(m))}}^m)$.
	Thus $\mathbb{B}\cap\Delta_{\alpha_{n(m)}^{(N(m))}}=\emptyset$.
	By (b), we have that
	$$P_{\alpha_{n(m)}^{(N(m))}}^{\comp n_0}(\mathbb{E})\cap\Delta_{\alpha_{n(m)}^{(N(m))}}=\emptyset$$
	and hence
	$$\left(\cup_{j=0}^{j=n_0}P_{\alpha_{n(m)}^{(N(m))}}^{\comp j}(\mathbb{E})\right)\cap\Delta_{\alpha_{n(m)}^{(N(m))}}=\emptyset.$$
	Then by (c) we have that
	$\cup_{j=0}^{j=n_0}P_{\alpha_{n(m)}^{(N(m))}}^{\comp j}(\mathbb{E})$ separates $\Delta_{\alpha_{n(m)}^{(N(m))}}$ and $\infty$.
	This implies that $\cup_{j=0}^{j=n_0}P_{\alpha_{n(m)}^{(N(m))}}^{\comp j}(\mathbb{E})$ separates the critical point of $P_{\alpha_{n(m)}^{(N(m))}}$ and $\infty$.
	This contradicts with (d) for large enough $m$.
	
	\vspace{0.2cm}
	\noindent{\bf Case 3:} Assume $\mathcal{D}_{\Gamma}\subseteq A_{\infty}$,
	where $A_{\infty}$ is the Fatou component of $P_{\alpha}$ containing $\infty$.
	Since the restriction of $P_{\alpha}$ to $A_{\infty}$ is conjugate to $z^2$ on the unit disk,
	it is easy to see that for large enough $t$, $P_{\alpha}^{\comp t}$ is not univalent on $\mathcal{D}_{\Gamma}$.
	Thus we get a contradiction.
	
	\vspace{0.2cm}
	\noindent{\bf Case 4:} Assume $\mathcal{D}_{\Gamma}\subseteq\Delta_{\alpha}$.
	Then $\Gamma\subseteq\overline{\mathcal{D}_{\Gamma}}\subseteq\overline{\Delta_{\alpha}}\subseteq\mathcal{O}_{P_{\alpha}}\cup\Delta_{\alpha}$.
	Since $\hat{I}_m(\alpha_{n(m)}^{(N(m))})+\gamma_m(\alpha_{n(m)}^{(N(m))})$
	converges uniformly to $\Gamma$ with respect to the spherical metric as $m\to+\infty$, we have that 
	$$\lim_{m\to+\infty}{\rm dist_H}\left(\hat{I}_m(\alpha_{n(m)}^{(N(m))})+\gamma_m(\alpha_{n(m)}^{(N(m))}), \mathcal{O}_{P_{\alpha}}\cup\Delta_{\alpha}\right)=0,$$
	which contradicts with (\ref{e7181}).
	
	\vspace{0.2cm}
	\noindent This completes the proof of the claim.
	It follows from the claim that
	\begin{equation}
		\label{e791}{\rm dist_H}\left(\partial\hat{\Delta}^m(\alpha_{n(m)}^{(N(m))}), \mathcal{O}_{P_{\alpha}}\cup\Delta_{\alpha}\right)\to0\ {\rm as}\ m\to+\infty.
	\end{equation}
	Next, we prove
	\begin{equation}
		\label{e792}{\rm dist_H}\left(\hat{\Delta}^m(\alpha_{n(m)}^{(N(m))}), \mathcal{O}_{P_{\alpha}}\cup\Delta_{\alpha}\right)\to0\ {\rm as}\ m\to+\infty.
	\end{equation}
	Let $\widehat{\mathcal{O}_{P_{\alpha}}\cup\Delta_{\alpha}}$ be the filling-in of $\mathcal{O}_{P_{\alpha}}\cup\Delta_{\alpha}$.
	We divide the proof of (\ref{e792}) into the following two steps:
	
	\vspace{0.2cm}
	\begin{itemize}
		\item[{\bf Step 1:}] we prove
		$${\rm dist_H}\left(\hat{\Delta}^m(\alpha_{n(m)}^{(N(m))}), \widehat{\mathcal{O}_{P_{\alpha}}\cup\Delta_{\alpha}}\right)\to0\ {\rm as}\ m\to+\infty;$$
		
		\vspace{0.2cm}
		\item[{\bf Step 2:}] we prove $\widehat{\mathcal{O}_{P_{\alpha}}\cup\Delta_{\alpha}}=\mathcal{O}_{P_{\alpha}}\cup\Delta_{\alpha}.$
	\end{itemize}
	
	\vspace{0.2cm}
	\noindent{\bf Step 1:}
	For any $r>0$, we denote by $\overline{\mathbb{U}}(\widehat{\mathcal{O}_{P_{\alpha}}\cup\Delta_{\alpha}},r)$ the closed $r$-neighborhood of $\widehat{\mathcal{O}_{P_{\alpha}}\cup\Delta_{\alpha}}$.
	Let $W$ be any connected component of $\mathbb{C}\setminus\overline{\mathbb{U}}(\widehat{\mathcal{O}_{P_{\alpha}}\cup\Delta_{\alpha}},r)$.
	We have the following two properties:
	\begin{itemize}
		\item[1)] For any $0<r_1\leq r$, $W$ is contained in some connected component of $\mathbb{C}\setminus\overline{\mathbb{U}}(\widehat{\mathcal{O}_{P_{\alpha}}\cup\Delta_{\alpha}},r_1)$.
		\item[2)] If $W$ is contained in some bounded connected component of $\mathbb{C}\setminus\overline{\mathbb{U}}(\widehat{\mathcal{O}_{P_{\alpha}}\cup\Delta_{\alpha}},r_1)$, then for any $r_2$ with $r_1\leq r_2\leq r$, $W$ is contained in some bounded connected component of $\mathbb{C}\setminus\overline{\mathbb{U}}(\widehat{\mathcal{O}_{P_{\alpha}}\cup\Delta_{\alpha}},r_2)$.
	\end{itemize}

	Next, we prove that for any $\epsilon>0$, there exists an $\epsilon_0$ with $0<\epsilon_0<\frac{\epsilon}{2}$ such that every bounded connected component of $\mathbb{C}\setminus\overline{\mathbb{U}}(\widehat{\mathcal{O}_{P_{\alpha}}\cup\Delta_{\alpha}},\frac{\epsilon}{2})$ is contained in the unbounded connected component of $\mathbb{C}\setminus\overline{\mathbb{U}}(\widehat{\mathcal{O}_{P_{\alpha}}\cup\Delta_{\alpha}},\epsilon_0)$.
	
	In fact, let $V$ be the union of all bounded connected components of $\mathbb{C}\setminus\overline{\mathbb{U}}(\widehat{\mathcal{O}_{P_{\alpha}}\cup\Delta_{\alpha}},\frac{\epsilon}{2})$ and for any point $o\in V$, we have that
	the Euclidean ball $\mathbb{B}(o,\frac{\epsilon}{2})$ centering at $o$ with radius $\frac{\epsilon}{2}$ does't intersect $\widehat{\mathcal{O}_{P_{\alpha}}\cup\Delta_{\alpha}}$.
	Let $o_1$ and $o_2$ be two different points in $V$. Then we have that
	\begin{itemize}
		\item[3)] if $|o_1-o_2|<\frac{\epsilon}{2}$, then $o_1$ and $o_2$ are in the same connected component of $\mathbb{C}\setminus\overline{\mathbb{U}}(\widehat{\mathcal{O}_{P_{\alpha}}\cup\Delta_{\alpha}},\frac{\sqrt{3}\epsilon}{8})$.
	\end{itemize}
	For any connected component $U$ of $\mathbb{C}\setminus\overline{\mathbb{U}}(\widehat{\mathcal{O}_{P_{\alpha}}\cup\Delta_{\alpha}},\frac{\sqrt{3}\epsilon}{8})$ with $U\cap V\not=\emptyset$,
	we take a point $o_U\in U\cap V$. Then
	we have a map $$\tau: U\mapsto\mathbb{B}(o_U,\frac{\epsilon}{4}).$$
	For any two different connected components $U$ and $U'$ of $\mathbb{C}\setminus\overline{\mathbb{U}}(\widehat{\mathcal{O}_{P_{\alpha}}\cup\Delta_{\alpha}},\frac{\sqrt{3}\epsilon}{8})$,
	if $U\cap V\not=\emptyset$ and $U'\cap V\not=\emptyset$, then by ${\rm 3)}$ we have $\tau(U)\cap\tau(U')=\emptyset$.
	Thus there are at most finitely many components of $\mathbb{C}\setminus\overline{\mathbb{U}}(\widehat{\mathcal{O}_{P_{\alpha}}\cup\Delta_{\alpha}},\frac{\sqrt{3}\epsilon}{8})$ intersecting $V$, written as $U_1, U_2,\cdots, U_k$.
	For any $j\in\{1,2,\cdots,k\}$, we claim that there exists an $\epsilon(j)$ with $0<\epsilon(j)<\frac{\sqrt{3}\epsilon}{8}$ such that $U_j$ is contained in the unbounded connected component of $\mathbb{C}\setminus\overline{\mathbb{U}}(\widehat{\mathcal{O}_{P_{\alpha}}\cup\Delta_{\alpha}},\epsilon(j))$. If not, $U_j$ will be contained in a bounded connected component of $\mathbb{C}\setminus\overline{\mathbb{U}}(\widehat{\mathcal{O}_{P_{\alpha}}\cup\Delta_{\alpha}},r)$ for any $0<r<\frac{\sqrt{3}\epsilon}{8}$. This implies that $U_j$
	is contained in a bounded connected component of $\mathbb{C}\setminus\widehat{\mathcal{O}_{P_{\alpha}}\cup\Delta_{\alpha}}$. This contradicts with the fact that $\widehat{\mathcal{O}_{P_{\alpha}}\cup\Delta_{\alpha}}$ is a filling-in.
	Set
	$$\epsilon_0=\min\{\epsilon(1),\epsilon(2),\cdots,\epsilon(k)\}.$$
	Then for any $j\in\{1,2,\cdots,k\}$, $U_j$ is contained in the unbounded connected component of $\mathbb{C}\setminus\overline{\mathbb{U}}(\widehat{\mathcal{O}_{P_{\alpha}}\cup\Delta_{\alpha}},\epsilon_0)$, and hence $V$ is contained in the unbounded connected component of $\mathbb{C}\setminus\overline{\mathbb{U}}(\widehat{\mathcal{O}_{P_{\alpha}}\cup\Delta_{\alpha}},\epsilon_0)$.
	
	By (\ref{e791}), there exists a positive integer $m_0$ such that for all $m\geq m_0$, we have $$\partial\hat{\Delta}^m(\alpha_{n(m)}^{(N(m))})\subseteq\overline{\mathbb{U}}(\widehat{\mathcal{O}_{P_{\alpha}}\cup\Delta_{\alpha}},\epsilon_0).$$
	Thus $\hat{\Delta}^m(\alpha_{n(m)}^{(N(m))})$ is contained in the filling-in of $\overline{\mathbb{U}}(\widehat{\mathcal{O}_{P_{\alpha}}\cup\Delta_{\alpha}},\epsilon_0)$.
	%which the union of $\overline{\mathbb{U}}(\widehat{\mathcal{O}_{P_{\alpha}}\cup\Delta_{\alpha}},\epsilon_0)$ and all connected components of $\mathbb{C}\setminus\overline{\mathbb{U}}(\widehat{\mathcal{O}_{P_{\alpha}}\cup\Delta_{\alpha}},\epsilon_0)$.
	It follows that
	\begin{equation}
		\label{e7101}\hat{\Delta}^m(\alpha_{n(m)}^{(N(m))})\cap V=\emptyset.
	\end{equation}
	
	At last, we will prove that $\hat{\Delta}^m(\alpha_{n(m)}^{(N(m))})\subseteq\overline{\mathbb{U}}(\widehat{\mathcal{O}_{P_{\alpha}}\cup\Delta_{\alpha}},\frac{\epsilon}{2})$ for all $m\geq m_0$.
	%and $\widehat{\mathcal{O}_{P_{\alpha}}\cup\Delta_{\alpha}}$ is less than or equal to $\frac{\epsilon}{2}$.
	In fact, if not, say, there exists a point $b\in\hat{\Delta}^m(\alpha_{n(m)}^{(N(m))})$ such that
	the Euclidean distance between $b$ and $\widehat{\mathcal{O}_{P_{\alpha}}\cup\Delta_{\alpha}}$ is greater than $\frac{\epsilon}{2}$, then
	$$b\in\mathbb{C}\setminus\overline{\mathbb{U}}(\widehat{\mathcal{O}_{P_{\alpha}}\cup\Delta_{\alpha}},\frac{\epsilon}{2})\subseteq\mathbb{C}\setminus\overline{\mathbb{U}}(\widehat{\mathcal{O}_{P_{\alpha}}\cup\Delta_{\alpha}},\epsilon_0)$$
	and $b$ belongs to the filling-in of $\overline{\mathbb{U}}(\widehat{\mathcal{O}_{P_{\alpha}}\cup\Delta_{\alpha}},\epsilon_0)$. Then $b$ belongs to some bounded connected component of $\mathbb{C}\setminus\overline{\mathbb{U}}(\widehat{\mathcal{O}_{P_{\alpha}}\cup\Delta_{\alpha}},\epsilon_0)$.
	It follows from ${\rm 2)}$ that $b\in V$. Thus $b\in V\cap\hat{\Delta}^m(\alpha_{n(m)}^{(N(m))})$, which contradicts with (\ref{e7101}).
	Then we have $\hat{\Delta}^m(\alpha_{n(m)}^{(N(m))})\subseteq\overline{\mathbb{U}}(\widehat{\mathcal{O}_{P_{\alpha}}\cup\Delta_{\alpha}},\frac{\epsilon}{2})$ for all $m\geq m_0$.
	Thus
	$${\rm dist_H}\left(\hat{\Delta}^m(\alpha_{n(m)}^{(N(m))}), \widehat{\mathcal{O}_{P_{\alpha}}\cup\Delta_{\alpha}}\right)\to0\ {\rm as}\ m\to+\infty.$$
	
	\noindent\vspace{0.2cm}
	{\bf Step 2:}
	We prove $\widehat{\mathcal{O}_{P_{\alpha}}\cup\Delta_{\alpha}}=\mathcal{O}_{P_{\alpha}}\cup\Delta_{\alpha}$ by contradiction.
	We suppose that there exists a nonempty bounded connected component of $\mathbb{C}\setminus(\mathcal{O}_{P_{\alpha}}\cup\Delta_{\alpha})$, written as $W_1$. Then $W_1$ is a Fatou component and $\partial W_1\subseteq \mathcal{O}_{P_{\alpha}}$.
	By Sullivan no wandering theorem, there exists a positive integer $t$ such that
	$P_{\alpha}^{\comp t}(W_1)=\Delta_{\alpha}$ and hence $P_{\alpha}^{\comp t}(\partial W_1)=\partial\Delta_{\alpha}$.
	It follows from [Theorem $11.2$ or its proof, \cite{DL}] that $P_{\alpha}$ is injective on $\mathcal{O}_{P_{\alpha}}$. Since $P_{\alpha}(\mathcal{O}_{P_{\alpha}})=\mathcal{O}_{P_{\alpha}}$, we have $P_{\alpha}^{\comp t}$ is also injective on $\mathcal{O}_{P_{\alpha}}$.
	Observe that $P_{\alpha}^{\comp t}(\partial\Delta_{\alpha})=\partial\Delta_{\alpha}$ and $\partial\Delta_{\alpha}\subseteq\mathcal{O}_{P_{\alpha}}$.
	Thus we have $\partial W_1=\partial\Delta_{\alpha}$.
	We take a point $o\in\partial W_1=\partial\Delta_{\alpha}$ that is not in the grand orbit of the critical point $c_0$. Then $P_{\alpha}^{\comp t}(o)\in\partial\Delta_{\alpha}$.
	Since $P_{\alpha}^{\comp t}$ is injective on $\mathcal{O}_{P_{\alpha}}\supseteq\partial\Delta_{\alpha}=\partial W_1$, we have that
	$$P_{\alpha}^{-t}(P_{\alpha}^{\comp t}(o))\cap\partial\Delta_{\alpha}=P_{\alpha}^{-t}(P_{\alpha}^{\comp t}(o))\cap\partial W_1=\{o\}.$$
	We take a sequence $\{z_j\}_{j=1}^{\infty}$ of complex numbers such that
	$z_j\in\Delta_{\alpha}, \forall j$ and $\lim_{j\to\infty}z_j=P_{\alpha}^{\comp t}(o)$.
	Then there exist $\{x_j\}_{j=1}^{\infty}\subseteq W_1$ and $\{y_j\}_{j=1}^{\infty}\subseteq\Delta_{\alpha}$ such that $P_{\alpha}^{\comp t}(x_j)=P_{\alpha}^{\comp t}(y_j)=z_j, \forall j$ and
	$\lim_{j\to\infty}x_j=\lim_{j\to\infty}y_j=o$.
	This implies that $o$ is a critical point of $P_{\alpha}^{\comp t}$,
	which contradicts with the fact that $o$ is not in the grand orbit of $c_0$.
	\begin{remark}
	{\rm In the proof of {\bf Step 2}, we use the result in \cite{DL}: $P_{\alpha}$ is injective on $\mathcal{O}_{P_{\alpha}}$. In fact, [Theorem $11.2$, \cite{DL}]
	told us that $P_{\alpha}$ is a homeomorphism on a Mother Hedgehog, by which a simpler proof of {\bf Step 2} can be given.}
	\end{remark}
	
	\vspace{0.2cm}
	It follows from {\bf Steps 1} and {\bf 2} that
	$${\rm dist_H}\left(\hat{\Delta}^m(\alpha_{n(m)}^{(N(m))}), \mathcal{O}_{P_{\alpha}}\cup\Delta_{\alpha}\right)\to0\ {\rm as}\ m\to+\infty.$$
	Since $\mathcal{O}_{P_{\alpha_{n(m)}^{(N(m))}}}\subseteq\hat{\Delta}^m(\alpha_{n(m)}^{(N(m))})$, we have that 
	$${\rm dist_H}\left(\mathcal{O}_{P_{\alpha_{n(m)}^{(N(m))}}}, \mathcal{O}_{P_{\alpha}}\cup\Delta_{\alpha}\right)\to0\ {\rm as}\ m\to+\infty.$$
	By the arbitrariness of $N(m)$ ($\gg m$), we can get
	$${\rm dist_H}\left(\mathcal{O}_{P_{\alpha_{n(m)}}}, \mathcal{O}_{P_{\alpha}}\cup\Delta_{\alpha}\right)\to0\ {\rm as}\ m\to+\infty.$$
	Then the first half part of the theorem follows immediately from the arbitrariness of $n(m)$ ($\gg m$).
	At last, by the same way as deriving (\ref{e792}) from (\ref{e791}), we can obtain the last half part of the theorem from the first half part.

\end{document}